\documentclass[final,onefignum,onetabnum]{siamart171218}

\usepackage{lipsum}
\usepackage{amsfonts}
\usepackage{graphicx}
\usepackage{epstopdf}
\usepackage{todonotes} 
\usepackage{booktabs}
\usepackage{bm}
\usepackage{algorithmic}
\ifpdf
  \DeclareGraphicsExtensions{.eps,.pdf,.png,.jpg}
\else
  \DeclareGraphicsExtensions{.eps}
\fi
\usepackage[right]{showlabels}

\newcommand{\TheTitle}{Bump attractors and waves in networks of leaky integrate-and-fire
neurons}
\newcommand{\TheShortTitle}{Bumps and waves in spiking networks}
\newcommand{\TheAuthors}{D. Avitabile, J.~L.~Davis, K.~C.~A.~Wedgwood}

\headers{\TheShortTitle}{\TheAuthors}

\title{{\TheTitle}}

\author{
  Daniele Avitabile
  \thanks{Department of Mathematics, Vrije Universiteit Amsterdam. 
    Amsterdam Neuroscience, Systems \& Network Neuroscience.
    MathNeuro Team, Inria Sophia Antipolis.
\email{d.avitabile@vu.nl},
    \url{https://www.danieleavitabile.com}.}
  \and
  Joshua L.~Davis
  \thanks{
    Defence Science and Technology Laboratory,
    Cyber and Information Systems Division.}
  \and
  Kyle C.~A.~Wedgwood
  \thanks{
    Living Systems Institute, College of Engineering, Mathematics and Physical Sciences.}
}

\usepackage{amsopn}
\usepackage{mathtools}

\renewcommand{\phi}{\varphi}

\DeclareMathOperator{\diag}{diag}

\DeclareMathOperator{\real}{Re}

\newcommand{\iunit}{\mathrm{i}\mkern1mu}

\newcommand{\RSet}{\mathbb{R}}
\newcommand{\DSet}{\mathbb{D}}
\newcommand{\NSet}{\mathbb{N}}
\newcommand{\ZSet}{\mathbb{Z}}
\newcommand{\SSet}{\mathbb{S}}
\newcommand{\FSet}{\mathbb{F}}

\newcommand{\CSet}{\mathbb{C}}
\newcommand\blank{{\mkern 2mu\cdot\mkern 2mu}}
\newcommand{\eps}{\varepsilon}

\newcommand{\tw}[1]{\ensuremath{\textrm{TW}_{#1}}}
\newcommand{\hb}[1]{\ensuremath{\textrm{HB}_{#1}}}

\newsiamremark{hypothesis}{Hypothesis}
\crefname{hypothesis}{Hypothesis}{Hypotheses}
\newsiamremark{problem}{Problem}
\crefname{problem}{Problem}{Problems}
\newsiamremark{remark}{Remark}
\crefname{remark}{Remark}{Remarks}

\newcommand{\otoprule}{\midrule[\heavyrulewidth]}

\graphicspath{{Figures/PDF/}}

\ifpdf
\hypersetup{pdftitle={\TheTitle},pdfauthor={\TheAuthors}}
\fi

\begin{document}
\maketitle

\begin{abstract}
  Bump attractors are wandering localised patterns observed in in vivo experiments of
  spatially-extended neurobiological networks. They are important for the brain's
  navigational system and specific memory tasks. A bump attractor is characterised by a core
  in which neurons fire frequently,
  while those away from the core do not fire. These structures have been found in
  simulations of spiking neural networks, but we do not yet have a mathematical
  understanding of their existence because a rigorous analysis of the nonsmooth
  networks that support them is challenging.
  We uncover a relationship between bump attractors and travelling waves in a
  classical network of excitable, leaky integrate-and-fire neurons. This relationship
  bears strong similarities to the one between complex spatiotemporal patterns and
  waves at the onset of pipe turbulence. 
  Waves in the spiking network are determined by a firing set, that is, the
  collection of times at which neurons reach a threshold and fire as the wave
  propagates.
  We define and study analytical properties of
  the
  \textit{voltage
  mapping}, an operator transforming a solution's firing set into its spatiotemporal
  profile. This operator
  allows us to construct localised travelling
  waves with an arbitrary number of spikes at the core, and to study
  their linear stability. 
%
  A homogeneous ``laminar" state exists in the network, and it is
  linearly stable for all values of the principal control parameter. Sufficiently
  wide disturbances to the homogeneous state elicit the bump attractor. 
  We show that
  one can construct waves with a seemingly arbitrary number
  of spikes at the core; the higher the number of spikes, the slower the
  wave, and the more its profile resembles a stationary bump. 
  \color{black}
  As in the
  fluid-dynamical analogy,
  such waves coexist with the homogeneous
  state, and the solution branches to which they belong are
  disconnected from the laminar state; 
  \color{black}
  we provide evidence that the
  dynamics of the bump attractor displays echoes of unstable waves, which form
  its building blocks.
\end{abstract}

\section{Introduction}\label{sec:intro}
\begin{figure}\label{fig:bumpSummary}
  \centering
  \includegraphics{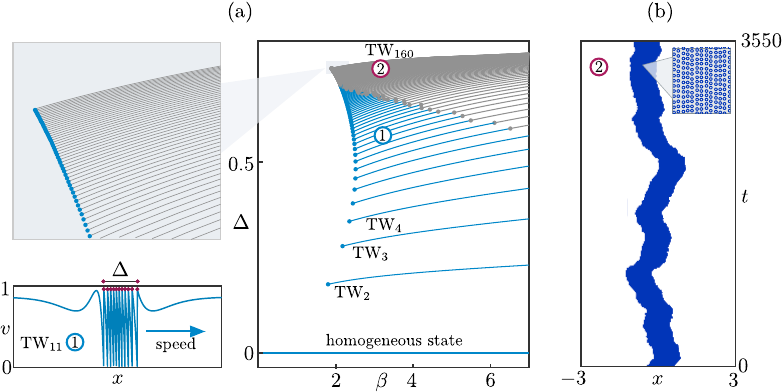}
  \caption{(a) Bifurcation diagram of travelling waves in a continuous
    integrate-and-fire model. (b) Bump attractor in a discrete integrate-and-fire
    model with 5000 neurons (dots represent neuronal firing events, \color{black}
    visible in the zoomed inset\color{black}). Model descriptions and
    parameters will be
    given later in \cref{sec:discreteModel,tab:parameters,fig:exampleBumps,fig:exampleWaves}). 
    The bifurcation diagram in (a) shows selected branches of stable (blue) and unstable (grey)
    travelling waves, in the continuation parameter $\beta$, that is, the timescale
    at which neuron process incoming currents. Waves are measured using their width
    $\Delta$ and are indexed by the number of advected spikes. The profile of
    \tw{11}, a representative wave with $11$ spikes, is shown. A large number of
    waves (\tw{2}--\tw{160} in the picture, but many more unstable branches are
    omitted) coexist with the trivial homogeneous
    state, which is the only steady state in the model, and which is stable for all
    values of $\beta$. Narrow waves are stable, with small basins of attraction.
    Sufficiently large, localised disturbances of the homogenous state lead to the
    formation of a bump with a characteristic width: the bump in (b) is marked as (2)
    in (a). The region in parameter space where bumps are
    observed is crowded with unstable travelling waves, with a large number of
    spikes, and a width comparable to the one of the bump. Branches of waves are
    detached from the homogeneous state; they originate at critical points called
    grazing points (blue dots in (a)); waves that are born stable become unstable at
    oscillatory bifurcations (grey dots in (a)).
}
\end{figure}
Understanding how networks of coupled, excitable units generate collective patterns is a
central question in the life sciences and, more generally, in applied mathematics.
In particular, the study of network models
is ingrained in neuroscience applications, as they provide a natural way to describe
the interaction of neurons within a population, or of neural populations within the
cortex. In the past decades, a large body of work in mathematical neuroscience has
addressed the development and analysis of neurobiological networks, with the view of studying the
origin of large-scale brain
activity~\cite{ermentrout2010mathematical,bressloff2014waves,coombes2014neural}, and
mapping single-cell and population parameters to experimental observations, 
including in vivo and in vitro cortical
waves~\cite{Richardson:2005cs,Huang:2004kw,GonzalezRamirez:2015gk},
electroencephalogram recordings~\cite{SteynRoss:2003ep}, and patterns in the visual
cortex~\cite{Camperi:1998ji}.
%
%

This paper presents a novel mathematical characterisation of a prominent example of
spatiotemporal pattern in neuroscience applications, and draws an analogy inspired by
recent progress in the fluid-dynamics literature on transition to turbulence in a
pipe~\cite{barkley2016theoretical}. We focus on the so-called \textit{bump
attractor}\footnote{In the neuroscience
  literature the term \textit{bump attractor} refers sometimes to a network producing a
localised pattern, as opposed to the pattern itself. Similarly, some authors use
\textit{ring attractor} for a network with ring topology, generating a localised
activity bump. Here, we use these terms to refer to patterns, following the standard
convention in the dynamical systems literature.}, a localised pattern of
neural activity observable in experiments and numerical simulations of
spatially-extended, neurobiological networks
\cite{redish1996coupled,zhang1996representation}. Bump attractors have been
associated to \textit{working
memory}, the temporary storage of information in the brain, and experimental evidence
supporting their existence has been found in the navigational systems of
rats~\cite{knierim2012attractor} and flies
\cite{kim2017ring,turner2017angular}, and in oculomotor responses of
monkeys~\cite{wimmer2014bump}.

In a bump attractor, the neural activity is localised around a particular position in the
network (see \cref{fig:bumpSummary}(b)) which may encode, for instance, the animal's head
position. Bumps are elicited by transient localised stimuli, such as visual cues at
specific locations, but are sustained
autonomously by the network once the stimulus is removed (the network dynamics is
\emph{attracted} to the bump). These coherent structures display a characteristic
wandering motion, and may exhibit discontinuous jumps if the impinging stimulus
undergoes sudden spatial shifts \cite{kim2017ring}.

\subsection{Model descriptions} Mathematical neuroscience has a long-standing
fascination with localised bumps of
activity. Neural field models, which represent the cortex as a continuum, were
introduced in the 1970s, and spatially-localised solutions to these models appeared 
already in seminal papers on the subject, by Wilson and
Cowan~\cite{wilson1973mathematical}, and by Amari~\cite{amari1977dynamics}. Since
then, many authors have studied localised solutions in neural fields, addressed the
derivation of neural field equations from first principles, their relevance to a
wide variety of neural phenomena, and their rigorous mathematical treatment. We refer
the reader to \cite{ermentrout2010mathematical,bressloff2014waves,coombes2014neural}
for exhaustive introductions on this topic.

Neural fields are integro-differential equations which model the cortex as an
excitable, spatially-extended medium. Mathematical mechanisms for pattern
formation in neural fields are similar to the ones found in other nonlinear media,
such as reaction-diffusion systems, albeit their analysis requires some
modifications because these models contain nonlocal operators. Stationary bumps form via
instabilities of the homogeneous steady state, and their profile depends strongly on the
coupling, which typically involves excitation on short spatial scales, and inhibition
on longer
scales~\cite{ermentrout2010mathematical,bressloff2014waves,coombes2014neural}. Neural
fields support
travelling bump solutions, as well as wandering bumps. The latter are obtained in
neural fields that incorporate stochastic terms deriving, for instance, from noisy
currents~\cite{kilpatrick2013wandering,maclaurin2019determination}.

Neural fields are heuristic, coarse grained models, hence they bypass microscopic
details that are important in bump attractors. For instance, the
neural firing rate, which is an emergent neural property and an observable in the
bump attractor experiments, is a prescribed feature in neural fields, hardwired in the model
through an ad-hoc firing-rate function. On the other hand, numerical simulations
of large networks of Hodgkin--Huxley-type neurons with realistic biological
details can display emergent neural firing, but their mathematical treatment is
challenging, and still under development~\cite{Folias.2010,baladron2012mean}.

\emph{Spiking neural networks} are intermediate, bottom-up models which couple
neurons with idealised dynamics. The
salient feature of spiking models is that the firing of a neuron is described as an
event, and no attempt is made to model the temporal evolution of the membrane
potential during and after the
spike~\cite{izhikevich2007dynamical,gerstner2014neuronal,bressloff2014waves}. Spiking
neural networks are specified by 3 main ingredients: (i) an ordinary differential
equation (ODE) for the membrane potential of each neuron; (ii) rules to define the
occurrence and effects of a spike; (iii) the network coupling. 

Since the introduction of the first single-cell spiking model by
Lapicque~\cite{lapicque1907recherches}, the so-called \emph{leaky integrate-and-fire
model}, more realistic variants have been proposed, and spiking neural networks
have become a widely adopted tool in theoretical
neuroscience~\cite{tuckwell1988introduction,Burkitt:2006,gerstner2014neuronal}.
In specific spiking models, analytical progress has been made for single neurons
and spatially-independent networks using coordinate
transformations~\cite{ermentrout1986parabolic,Mirollo:2006ft}, dimension
reduction~\cite{luke2013complete,montbrio2015macroscopic}, and
probabilistic methods~\cite{delarue2015global} (see also the reviews
\cite{sacerdote2013stochastic,bick2019understanding}).
Exact mean-field reductions, amenable to standard pattern-formation
analysis, have been derived in selected spatially-extended
networks~\cite{laing2015exact,esnaola2017synchrony,byrne2019next,schmidt2020bumps},
but generally the study of bumps in spiking models has been possible only with
numerical simulations~\cite{Laing:2001fc,compte2000synaptic}.

The present paper investigates localised patterns supported in discrete and
continuous networks of nonlocally coupled leaky integrate-and-fire neurons. In direct numerical
simulations, we use a well-known discrete model, proposed by Laing and
Chow~\cite{Laing:2001fc}, whose details will be
given later. For now it will suffice to consider a cursory formulation of
the model, simulated in \cref{fig:bumpSummary}(b). The network describes the idealised,
dimensionless voltage dynamics of $n$ all-to-all coupled neurons, evenly-spaced in a
cortex with ring geometry,
\begin{equation}\label{eq:toyNetwork}
  \dot v_i = -v_i + I_i(t) + \sum_{j=1}^n S_{ij}(v_j,\beta), 
  \qquad i = 1,\ldots,n.
\end{equation}
The dynamics of the $i$th neuron's membrane voltage is specified in terms of an Ohmic
leakage current $-v_i$, an external current $I_i(t)$, and voltage-dependent currents,
received from other neurons via synaptic connections; the latter currents, indicated
by $S_{ij}$, have a characteristic time scale $\beta$, and are caused by $v_j$
crossing a fixed threshold (when the $j$th neuron \textit{fires}).  After a firing event,
marked with a dot in \cref{fig:bumpSummary}(b) \color{black} and its
inset\color{black}, the neuron's voltage is instantly reset to a
default value, from which it can evolve again, following an ODE of
type~\cref{eq:toyNetwork}.
Discrete and continuous networks of this type are canonical
models of neural activity, widely adopted in the mathematical neuroscience
literature~\cite{%
vanVreeswijk:1994fb,
vanVreeswijk:1996jf,
ermentrout1998neural,
Ermentrout1998c,
Bressloff:2000uj,
Laing:2001fc,
Ermentrout:2002dl,
Osan:2002jq,
compte2003cellular,
Osan:2004ko,
Mirollo:2006ft,
Gerstner:2008be}. 
It is now established that such networks support bump attractors and localised
waves, but an explanation of the mathematical origins of the former is still lacking.

This paper presents a new approach to the problem, 
and uncovers a novel bifurcation structure for localised travelling waves of the
network, shedding light onto the nature of the bump attractor. Our findings
suggest an intriguing analogy between the bump attractor in the integrate-and-fire
network and the phenomenon of transition to turbulence in a pipe. The analogy between
the bifurcation scenarios of these two problems is notable, and we use it here to
summarise our results, highlighting similarities between the respective bifurcation
structures and dynamical regimes.

\subsection{Transition to turbulence in a pipe}
Stemming from the pioneering experiments of Reynolds~\cite{reynolds1883xxix}, a large
body of work in fluid dynamics has addressed how high-speed pipe flows transition from a
laminar state, whose analytical expression is known in closed form, to complex
spatio-temporal patterns, characteristic of the turbulent regime
(see~\cite{barkley2016theoretical} for a recent review).
In this context, the Navier-Stokes equations are studied as a deterministic dynamical
system, subject to changes in Reynolds number, the principal control parameter.
Experiments and computer simulations indicate that the laminar state is stable to
infinitesimal perturbations (linearly stable) up to
large values of the control parameter (up to at least Reynolds number $10^7$ in
numerical computations)~\cite{
salwen1980linear,
darbyshire1995transition,
meseguer2003linearized,
van2009flow,
manneville2015transition}.
However, when a disturbance is applied at sufficiently large Reynolds
numbers, a transition to turbulence is observed, depending sensitively on
the applied stimulus \cite{darbyshire1995transition,hof2003scaling}. Current opinions
view the transition as being determined by \emph{travelling wave solutions} to the
Navier-Stokes equations
\cite{schmiegel1997fractal,echhardt2002turbulence,faisst2003traveling,
wedin2004exact,pringle2009highly,gibson2009equilibrium}. 
These invariant states, whose spatial
profiles display
hallmarks of the turbulent transition, coexist with the laminar state at intermediate
Reynolds numbers, are linearly unstable, and provide an intricate blueprint for the
dynamics, in that orbits may visit transiently these repelling solutions in phase
space. Importantly, the waves lie on branches that are disconnected from the stable
laminar state, and emerge at saddle-node bifurcations
\cite{faisst2003traveling,wedin2004exact}: this turbulence mechanism is therefore
different from other paradigmatic routes to chaos, involving the destabilisation of
the laminar state, and the progressive appearance of more complicated structures via
a cascade of
instabilities~\cite{landau1944problem,hopf1948mathematical,ruelle1971nature}. 

\subsection{Summary of results} In a series of recent papers addressing turbulence
from a dynamical-system viewpoint, Barkley proposed an analogy between pipe
flows and excitable media, using the propagation of an electrical pulse along the
axon of a neuron as a metaphor for localised turbulence puffs
\cite{barkley2011simplifying,barkley2012pipe,barkley2015rise,barkley2016theoretical}.
The present paper offers a specular view, at a different scale: we are motivated by
studying a canonical, complex neurobiological network of coupled excitable neurons,
supporting localised spatio-temporal chaos, and we find a compelling similarity
between the bifurcation structure of waves in this system, and the one of waves in
the pipe turbulence.

With reference to \cref{fig:bumpSummary}, the principal control parameter of the
problem is $\beta$, the timescale of synaptic currents: a low $\beta$ gives small,
persisting currents, while $\beta \to \infty$ gives large instantaneous currents. A
homogeneous steady state 
exists and is linearly stable for all values of $\beta$ ($\Delta =0$
line in \cref{fig:bumpSummary}(b)), but transient localised stimuli trigger the bump
attractor~\cite{Laing:2001fc}. In the analogy, the homogeneous equilibrium plays
the role of a ``laminar state". We stress that the homogeneous steady state
is the only equilibrium of the model. Thus, the model can not support branches of
stationary bump solutions. Instead, we demonstrate that travelling waves are key to
understand the bump attractor.

We consider a spatially-continuous version of model~\cref{eq:toyNetwork}, which is known to support
waves advecting a low number of \emph{localised spikes}, or having a non-localised
profile~\cite{Ermentrout1998c,Bressloff:2000dq,Bressloff:2000uj,Ermentrout:2002dl,Osan:2002jq}.
The travelling waves of interest to us, however, have a localised profile, and advect a
\emph{large} number of spikes, such as the one presented in
\cref{fig:bumpSummary}(a). These structures are not accessible with the
current techniques, hence we develop here analytical and numerical tools to construct
them.
We define particular type of solutions, which retain a fixed number of spikes in
time; this class of solutions is sufficiently general to incorporate
travelling waves with an arbitrary, finite, number of spikes, and small perturbations to
them. We introduce the \emph{voltage mapping}, a new operator which formalises an
idea previously used in the literature for
spiking~\cite{Ermentrout1998c,Bressloff:2000dq,Bressloff:2000uj,Ermentrout:2002dl,Osan:2002jq,Avitabile2017}
and non-spiking networks
\cite{amari1977dynamics,ermentrout2010mathematical,bressloff2014waves,coombes2014neural}.
The voltage mapping is based on level sets describing firing events,
and it allows efficient travelling wave constructions and stability computations. 

Using the voltage mapping,
we construct numerically waves with more than $200$ concurrent spikes. These waves are
spatially localised, and
coexist with the trivial (laminar) state (see \cref{fig:bumpSummary}(a)); 
\textcolor{black}{most of the
waves we computed are unstable, and the stable ones have a small basin of attraction}. As in the
turbulence analogy, the waves contain features of the bump attractor: they pack a
seemingly arbitrary number of
spikes within the width of a bump attractor, and they advect them at an
arbitrarily slow speed, depending on $\beta$, and on the number of carried spikes. As
in the fluid-dynamical analogy, waves are disconnected from the laminar state.
Owing to the intrinsic non-smoothness of the network, the waves emerge primarily at
grazing points \textcolor{black}{(as opposed to the saddle-node bifurcations seen in
the fluid-dynamical analogy, and also observed here in certain parameter regimes).}
In addition, we present numerical evidence that the \textcolor{black}{transient
dynamics to} the bump attractor displays echoes of the unstable waves which, as in
the fluid-dynamics analogy, form building blocks for the localised structure. Also,
the characteristic wandering of the bump attractor, whose excursions become more
prominent as $\beta$ increases, is supported by this purely \emph{deterministic}
system, \color{black} akin to the pseudo-stochastic behaviour observed in balanced
neural networks~\cite{Vreeswijk.1998,Litwin-Kumar.2012,Rosenbaum.2014}.\color{black}

The paper is structured as follows: in \cref{sec:discreteModel} we introduce the
discrete model, characterise it as a non-smooth threshold network, and present
numerical simulations of bumps and waves; in \cref{sec:TWm} we introduce the
continuum model, the voltage mapping, and the construction of travelling waves; in
\cref{sec:TWstability} we discuss travelling wave stability, we present
numerical
results in \cref{sec:bif-structure-TW}, and we conclude in \cref{sec:conclusions}.

\section{Coherent structures in the discrete model}\label{sec:discreteModel}
We begin by introducing the discrete model by Laing and Chow~\cite{Laing:2001fc}. We
characterise it as a piecewise-linear dynamical system, and we show numerical
simulations of coherent structures. An important difference from the work by Laing
and Chow is that we consider a \emph{deterministic}
model, which we call the Discrete Integrate-and-Fire Model (DIFM). We remark that the
neurons considered here, taken in isolation, are in an \textit{excitable regime},
that is, they exhibit an all-or-none response, based on the input they receive. This
is considerably different from the so-called \textit{oscillatory regime}, in which
neurons, when decoupled from the network, display oscillations~\cite{vanVreeswijk:1994fb,Bressloff:2000uj,Mirollo:2006ft,Gerstner:2008be}.

\subsection{Description of the DIFM} 
\begin{figure}\label{fig:schematic}
  \centering
  \includegraphics[width=1.0\textwidth]{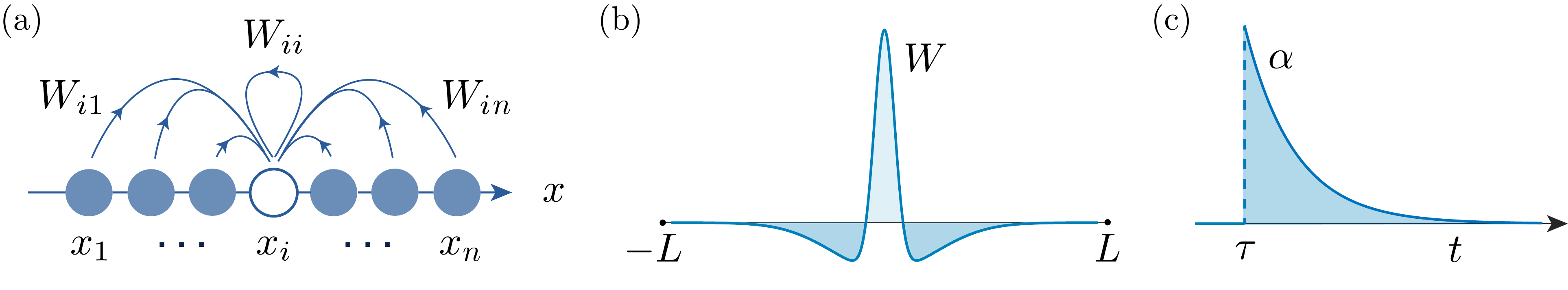}
  \color{black}
  \caption{(a): Schematic of the all-to-all coupled neurons with strengths $W_{ik} =
    w(|x_i-x_k|)$. In the model, we use a ring geometry, hence the left neighbour of
    $x_1$ is identified with $x_n$, and the right neighbour of $x_n$ with $x_1$. (b):
    The coupling (synaptic) function $w(x)$ is chosen to be $2L$-periodic and
    positive (excitatory) on short spatial scales, and negative
    (inhibitory) on long spatial scales. (c) Time-dependent neuronal (post-synaptic)
    currents are modelled via the function $\alpha$, which is null before a
  neuron fires $(t < \tau)$, and exponentially decaying thereafter (see
Equation~\cref{eq:alphaAndW}).}
\end{figure}
  \color{black}
The DIFM is a spatially-extended system of $n$ identical \emph{integrate-and-fire} neurons, posed on $\SSet = \RSet/2L\ZSet$,
that is, a ring of period $2L$. Neurons are indexed using the set $\NSet_n =
\{1,\ldots,n\}$ and occupy the discrete, evenly spaced nodes 
$x_i = -L + 2iL/n \in \SSet$, for $i \in \NSet_n$.
Neurons are coupled via their synaptic connections, which are modelled by a
continuous, bounded, even and exponentially decaying function $w \colon \SSet \to
\RSet$:
the strength of the connections from the $k$th to the $i$th neuron depends solely on
the distance $|x_i-x_k|$, measured around the ring, hence we write it as $W_{ik}=
w(x_i - x_k)$, for all $i,k \in \NSet_n$ (see \cref{fig:schematic}). 
\color{black}
We note that $w$ is $2L$-periodic by definition.
\color{black}

To the $i$th neuron is associated a real-valued time-dependent voltage function
$v_i(t)$, and the coherent structures of
interest are generated when voltages $\{v_i\}$ attain a threshold value (when neurons
\emph{fire}). The DIFM is formally written as follows:
\begin{align}
   \dot v_i(t) &= I_i(t)-v_i(t) +
   \frac{2L}{n}\sum_{k \in \NSet_n} \sum_{j \in \NSet} W_{ik} \alpha(t-\tau_{k}^{j})
   -\sum_{j \in \NSet} \delta (t-\tau^{j}_{i}),
    & i \in \NSet_n,
    \label{eq:vodes} \\
    v_i(0) &= v_{0i}, & i \in \NSet_n.
    \label{eq:inicond}
\end{align}

At time $\tau_i^j$, when the voltage $v_i$ reaches the value $1$ from below
for the $j$th time, a firing event occurs; a more precise definition of these
\emph{spiking times} will be given below. The formal evolution
equation~\cref{eq:vodes} expresses the modelling assumption that, when a neuron
fires, its voltage is instantaneously reset to $0$ (hence the Dirac delta), and a
so-called \emph{post-synaptic current} is received by all other neurons in the network,
with intensity proportional to the strength of the synaptic connections. The
time-evolution of this current is modelled via the \emph{post-synaptic function} $\alpha(t) = p(t) H(t)$, expressed as the product of a continuous potential
function $p$ and the Heaviside function $H$, hence the post-synaptic current is zero
before a spike.

In this paper, we present concrete calculations for
\begin{equation}\label{eq:alphaAndW}
  \alpha(t)=\beta \exp(-\beta t)H(t), \qquad
  w(x)= a_1 \exp(-b_1 |x|) -  a_2\exp(-b_2|x|),
\end{equation}
with $\beta, a_1, a_2, b_1, b_2 >0$, albeit the analytical and numerical framework
presented below is valid for more generic choices, subject to general assumptions
which will be made precise in \cref{subsec:TWCharacterisation}. The function $\alpha$ models exponentially-decaying currents with
rate $-\beta$ and initial value $\beta$, hence the limit $\beta \to \infty$ approximates
instantaneous currents. Currents with an exponential rise and decay are also
used in literature. The synaptic coupling function $w$ is chosen so that connections are
positive (\emph{excitatory}) on the lengthscale $1/b_1$, and negative 
(\emph{inhibitory}) on the lengthscale $1/b_2$ (see \cref{fig:schematic}). 

In addition to the post-synaptic current, neurons are subject to an external stimulus
$I_i(t)$. In certain time simulations, coherent structures will be elicited with the
application of a transient, heterogeneous stimulus of the form
\begin{equation}\label{eq:stimulus}
  I_i(t) = I + d_1 H(\tau_\textrm{ext}-t)/\cosh(d_2 x_i), \quad i \in \NSet_n.
\end{equation}
Our investigation, however, concerns asymptotic states of the autonomous homogeneous
case $I_i(t) \equiv I$, hence one should assume $d_1 =0$, unless stated otherwise.  A
description of model parameters and their nominal values can
be found in ~\cref{tab:parameters}.


\subsection{Event-driven DIFM}
Laing and Chow studied and simulated a stochastic version of the DIFM, using the Euler
method and a first-order interpolation scheme to obtain the firing
times~\cite{Laing:2001fc}. We use here a different approach: in preparation for
our analytical and numerical treatment of the problem, we write the formal model
\crefrange{eq:vodes}{eq:inicond} as a system of $2n$ piecewise-linear ODEs. 
To this end we introduce the synaptic input variables
\begin{equation} \label{eq:xsub} 
  s_{i}(t)= \frac{2L}{n} \sum_{k \in \NSet_n} \sum_{j \in \NSet} W_{ik} \alpha(t-\tau_{k}^{j}),
  \qquad i \in \NSet_n
\end{equation}
and combining~\cref{eq:alphaAndW} and~\cref{eq:vodes} we obtain formally
\[
\begin{aligned}
   \dot v_i(t) &= I_i(t)-v_i(t) + s_i(t)
   -\sum_{j \in \NSet} \delta (t-\tau^{j}_{i})
    \\
   \dot s_i(t) &= -\beta s_i(t) +
   \frac{2L\beta}{n}\sum_{k \in \NSet_n} \sum_{j \in \NSet} W_{ik} \delta(t-\tau_{k}^{j})
 \end{aligned} \qquad i \in \NSet_n.
\]
One way to define the associated non-smooth dynamical system is to express the model
as an impacting system, by partitioning the phase space $\RSet^{2n}$ via a switching
manifold, on which a reset map is prescribed (see \cite{diBernardo:2008cu} and
references therein for a discussion on non-smooth and impacting systems). Here, we
specify the dynamics so as to expose the firing times $\{\tau^j_k\}$, as opposed to
the switching manifold: this is natural in the mathematical neuroscience context, and
it prepares our analysis of the continuum model. Since $\{\tau^j_k\}$ are the times
at which
orbits in $\RSet^{2n}$
reach the switching manifold, a translation between the two formalisms is possible.

Following these considerations, we set $\tau^0_i = 0$ for all $i \in \NSet_n$,
introduce the notation $f(\blank^\pm) = \lim_{\mu \to 0^+} f(\blank\pm\mu)$, and
define firing times as follows\footnote{
  Note that $\{\tau_i^0\}_i$ are not firing times, but auxiliary symbols for
  the definition of firing times~\cref{eq:firingTimes}. Indeed, since the sums in
\cref{eq:vodes} run for $j \in \NSet$, the $\{\tau_i^0\}_i$ are immaterial for the
dynamics.}
\begin{equation}\label{eq:firingTimes}
 \tau_i^j = 
  \inf 
  \big\{
    t \in \RSet \colon t > \tau_i^{j-1}, \;
    v_i(t^-) = 1, \;
    \dot v_i(t^-) > 0
  \big\},
  \qquad
  i \in \NSet_n, \quad j \in \NSet.
\end{equation}
We arrange firing times in a monotonic increasing sequence
$\{\tau_{i_k}^{j_k}\}_{k=1}^q$ such that 
\begin{equation}\label{eq:timePartition}
(0,T] = 
\bigcup_{k \in \NSet_{q+1}} \big(\tau^{j_{k-1}}_{i_{k-1}},\tau^{j_k}_{i_k} \big],
\qquad
0 = \tau^{j_0}_{i_0}< \tau^{j_1}_{i_1} \leq \ldots 
\leq \tau_{i_{q}}^{j_{q}} < \tau_{i_{q+1}}^{j_{q+1}} = T,
\end{equation}
for some time horizon $T > 0$, and obtain the desired set of $2n$
piecewise-linear ODEs
\begin{equation}\label{eq:eventODE}
  \dot v_i  = I_i-v_i + s_i, \quad \dot s_i  = -\beta s_i  
  \qquad i \in \NSet_n, 
  \qquad t \in 
  \bigcup_{k \in \NSet_{q+1}} \big(\tau^{j_{k-1}}_{i_{k-1}},\tau^{j_k}_{i_k} \big], 
\end{equation}
with initial and reset conditions
\begin{align} 
  v_i(0) & = v_{0i}, 
     & s_i(0) &= s_{0i},  
     & i \in \NSet_n, &
     & & 
   \label{eq:eventICs}
     \\
  v_{i_k}(\tau^{j_k\,+}_{i_k}) &= 0, 
  & s_l(\tau^{j_k\,+}_{i_k})  & = s_l(\tau^{j_k \, -}_{i_k}) + \frac{2L\beta}{n}W_{li_k}, 
  &l \in \NSet_n, &
  & & k \in \NSet_q,
\label{eq:resets}
\end{align}
respectively. Henceforth, we refer to the non-smooth dynamical system
\crefrange{eq:firingTimes}{eq:resets} with connectivity function $w$ given
by~\cref{eq:alphaAndW} and stimulus~\cref{eq:stimulus} as the \emph{event-driven
DIFM} or simply \emph{DIFM}, that is, we view this model as a substitute for the formal
system~\crefrange{eq:vodes}{eq:inicond}.

Even though the firing-time notation may seem cumbersome at first, the evolution of
the DIFM is remarkably simple: \Cref{eq:eventODE} states that between two
consecutive firing times, neurons evolve independently, subject to a linear ODE; a solution in closed form can be
written in terms of exponential functions,
parametrised by the firing times. Constructing a solution amounts to determining
firing times (impacts with the switching manifold), as is customary in piecewise-linear
systems. This aspect will be a recurring theme in the sections analysing
travelling waves in the continuum model.
\begin{figure}
  \centering
  \includegraphics{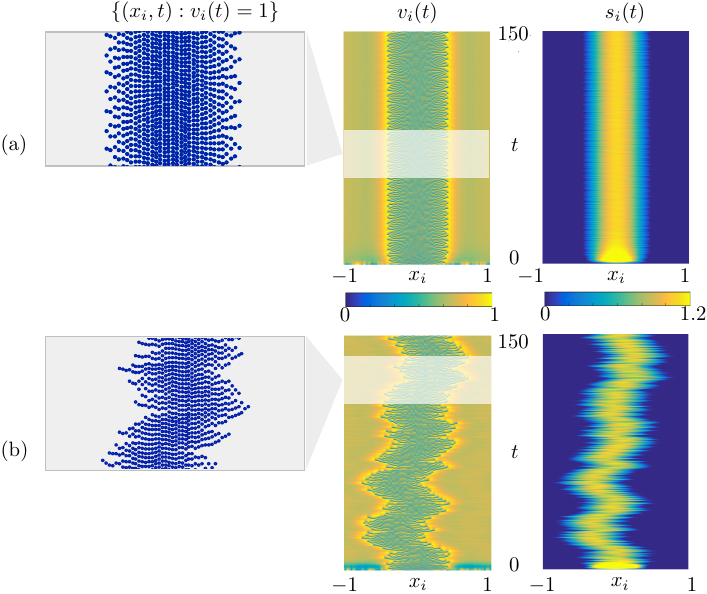}
  \caption{Bump attractors obtained via direct numerical simulation of the DIFM
    \crefrange{eq:firingTimes}{eq:resets} with external input~\cref{eq:stimulus} and
    connectivity function $w$ as in~\cref{eq:alphaAndW}.
    We visualise the network voltage (centre) and synaptic current (right) as
    functions of space and time and, in the inset (left), a raster plot of the
    firing events. Parameters as in~\cref{tab:parameters} with $n=80$, $d_1=2$ $d_2=10$. The
    network's synaptic time scale is $\beta = 1$ (a) and $\beta = 3.5$ (b),
    respectively. A localised coherent structure is visible in (a), which wanders
    when $\beta$ is increased. We remark that the system under consideration is
    deterministic.}
  \label{fig:exampleBumps}
\end{figure}

In simulations of the DIFM, we time step \Cref{eq:eventODE} rather than using its
analytic solution. We use an explicit adaptive 4-5th order
Runge-Kutta pair with continuous output, and detect events (compute firing times) by
root-finding \cite{Dormand1980,shampine1997matlab}. The simulation stops at each
firing event and is restarted after the reset conditions~\cref{eq:resets} are applied.
Simulating the event-driven DIFM instead of
\cref{eq:vodes} allows us to compute firing
times accurately, and to evolve the system without storing in memory or truncating
the synaptic input sums in~\cref{eq:vodes}. 

\subsection{Coherent structures in the DIFM}
\begin{figure}
  \centering
  \includegraphics{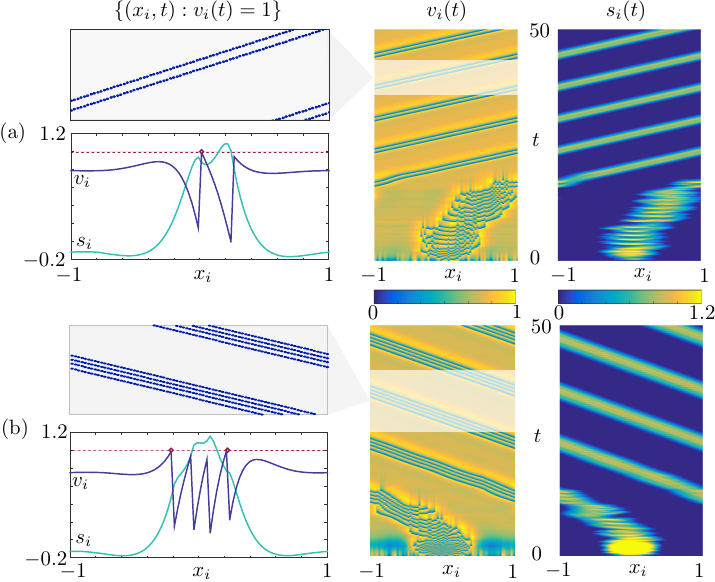}
  \caption{Stable coexistent waves obtained via direct numerical simulation of
    DIFM \crefrange{eq:firingTimes}{eq:resets} with external input~\cref{eq:stimulus}
    and connectivity function $w$ as in~\cref{eq:alphaAndW}. Parameters as
    in~\cref{tab:parameters} with $n=80$, $\beta = 4.5$ for both (a) and (b), but different
    initial stimuli: (a) $d_1=0.4$, $d_2=12$, (b) $d_1 = 2$, $d_2=10$.  Depending on
    the transient stimulus the model displays: (a) a wave
    propagating with positive speed, in which pairs of neurons fire asynchronously,
    but at short times from each other; (b) a similar structure
    involving a quartet of neurons. Coexisting structures with variable numbers
    of firing neurons have also been found (not shown). The spatial profiles
  indicate that neurons reach threshold (dashed red line) one at a time within a
pair (a) or two at a time within a quartet (b).}
  \label{fig:exampleWaves}
\end{figure}
The DIFM supports standing and travelling localised structures, as in the stochastic
setting~\cite{Laing:2001fc}. Bumps form robustly when we prescribe
homogeneous initial conditions\footnote{Typically we set $v_{0i} = u \in (0,1),
  s_{0i} = 0$, for $i \in \NSet_n$, but the coherent structures discussed in the
  paper can also be found with random, independent and identically distributed
initial voltages, for instance $v_{0i} \sim \mathcal{U}([0,1])$, where $\mathcal{U}$ is the uniform
distribution.}
with a short transient stimulus
(\Cref{eq:stimulus} with $\tau_{\textrm{ext}} = 2$). Since $I_i(t) \equiv I$ for
all $t>\tau_{\textrm{ext}}$, the structures observed over long-time intervals are
solutions to a homogeneous, non-autonomous problem.

As seen in \cref{fig:exampleBumps},
the bump wanders when $\beta$ is increased.
In passing, we note that this phenomenon is not due to stochastic effects, as studied
in other contexts
\cite{kilpatrick2013wandering,inglis2016general,Avitabile2017}, because the DIFM is
deterministic. For sufficiently
large $\beta$, the system exhibits stable travelling structures: in
\cref{fig:exampleWaves} we show two coexisting waves, found for $\beta = 4.5$ upon
varying slightly the width $d_1$ and intensity $d_2$ of the transient
stimulus. In each case we plot the voltage and synaptic profiles, and associated
raster plots. We notice different firing patterns in the waves, involving $2$ and $4$
firings, respectively: the wave with $2$ firings travels faster, and its voltage and
synaptic profiles are narrower. We found coexisting waves with a greater number of
firings and progressively lower speed, whose existence and bifurcation
structure will be at the core of the following sections.

\subsection{Remarks about coherent structures in the DIFM} The patterns presented so
far are found in the DIFM with a finite number of neurons $(n=80)$. At first sight,
the raster plots of the waves seem to indicate that neurons fire simultaneously in
pairs (\cref{fig:exampleWaves}(a)) or quartets (\cref{fig:exampleWaves}(b)) as the
structure travels across the network. A closer inspection of the instantaneous
profiles $v_i(t)$ reveal that this is not the case, as the threshold (red dashed
line) is attained by a single neuron in~\cref{fig:exampleWaves}(a), and by two
neurons in~\cref{fig:exampleWaves}(b): 
\color{black}
neurons in a raster pair fire alternately over a
short time interval, 
\color{black}
whereas a quartet displays a more complex firing pattern.

Hence, for finite $n$, the propagating structures displayed in~\cref{fig:exampleWaves} are
not strictly travelling waves, in the sense that the profile is not stationary in the
comoving frame; their dynamics is that of saltatory waves
\cite{Coombes:2003ec,Wolfrum:2015fu,Avitabile2017}. The saltatory nature of the
waves, however, is an effect of the network size: as we increase $n$, the amplitude
of temporal oscillations in the comoving frame scales as $O(n^{-1})$, and the
spatio-temporal profile converges to one of a travelling wave as $n \to \infty$. 

In addition, the structure in \cref{fig:exampleBumps}(a) is not a bump, in the sense
that it is not a spatially heterogeneous steady state of the DIFM, because the
pattern is sustained by firing events 
\color{black}
(and the presence of firing events means the voltage changes in time).
\color{black}
Indeed, the only equilibrium supported by the DIFM is the
homogeneous state $v_i(t) \equiv I$, $s_i(t)\equiv0$, $i \in \NSet_n$, which is
linearly stable for all values of $\beta$, as can be deduced by inspecting
system~\cref{eq:eventODE}.

\color{black}
By constructing travelling waves and investigating their stability
in a continuum version of the DIFM, we shall see that the structure
in \cref{fig:exampleBumps}(a) (and its wandering) can be interpreted as deterministic
chaotic behaviour.
\color{black}

\section{Travelling waves in the continuum model}\label{sec:TWm}
As stated in \cref{sec:discreteModel}, 
the profiles $\{v_i(t)\}_i$ and $\{s_i(t) \}_i$ in \cref{fig:exampleWaves} behave
like travelling wave solutions as $n \to \infty$. Motivated by this observation, we
study travelling waves in a continuum, translation-invariant version of the DIFM: we
set $d_1=0$ in the stimulus~\cref{eq:stimulus}, consider a continuum spatial domain,
and pose the model on $\RSet$ as opposed to $\SSet$, obtaining
\begin{equation} \label{eq:contMod}
  \begin{split} \partial_t v(x,t) = -
    v(x,t) + I 
        & + \sum_{j \in \NSet} \int_{-\infty}^\infty w(x-y) \alpha\big( t - \tau_j(y) \big) \, dy \\ 
	& - \sum_{j \in \NSet} \delta \big( t - \tau_j(x) \big), \qquad (x,t) \in \RSet \times
    \RSet.
  \end{split}
\end{equation}
The formal evolution equation presented above, which we henceforth call the
\emph{continuous integrate-and-fire model} (CIFM), has been proposed and studied by
several authors in the mathematical neuroscience literature
\cite{Ermentrout1998c,Golomb:1999cr,Bressloff:1999ik,Bressloff:2000dq,Osan:2002jq,Osan:2004ko}.  
In the CIFM, firing-time functions  $\tau_j(x)$ indicate that the neural patch at
position $x$ fires for the $j$th time, and replace the discrete model's firing
times $\tau_k^j$.\footnote{The index $j$ is used
as a superscript in the firing times, but for notational convenience we use it
as a subscript in the firing functions, so that $\tau_j(x_k) \approx \tau_k^j$.}
A graph of the firing functions replaces the raster plot in the discrete model, so
that a travelling wave in the CIFM corresponding to the $n \to \infty$ limit of the
structure in~\cref{fig:exampleWaves}(a), for instance, will involve $2$ linear firing
functions $\tau_1$, $\tau_2$, with $\tau_1(x)<\tau_2(x)$ for all $x \in \RSet$. 

The existence of travelling waves solutions in \cref{eq:contMod} with a single spike
has been studied
by Ermentrout \cite{Ermentrout1998c} who presented various scalings of the
wavespeed as a function of control parameters. A general formalism for the
construction and linear stability analysis of \emph{wavetrains} (spatially-periodic
travelling solutions) was introduced and analysed by Bressloff
\cite{Bressloff:2000dq}, who derived results in terms of Fourier series expansions.
The construction of travelling waves with multiple spikes was later studied by
O\c{s}an and coworkers~\cite{Osan:2004ko}, albeit stability for these states was not
presented and computations were limited to a few spikes, for purely excitatory
connectivity kernels. The common thread in the past literature on this topic is the idea
that travelling wave
construction and stability analysis rely entirely on knowledge of the firing function
$\tau_j$ (as in the DIFM, with firing times). A similar
approach has been used effectively in Wilson-Cowan-Amari neural field equations,
where it is often called \emph{interfacial dynamics} (see~\cite{amari1977dynamics} for the first study
of this type, \cite{Coombes:2014aa} for a recent review, and
\cite{coombes2011pulsating,folias2004breathing}, amongst others, for examples of
spatio-temporal patterns analysis).

Here we present a new treatment of travelling wave solutions that draws from this
idea; we introduce an operator, that we call the \emph{voltage mapping}, with the
following aims: (i)
Expressing a mapping between firing functions and
solution profiles, with the view of replacing the formal evolution equation
\cref{eq:contMod} for travelling waves with $m$ spikes (where $m$ is arbitrary).
(ii) Finding conditions for the linear stability of these waves. (iii) Using
root-finding algorithms to compute travelling waves and study their linear stability.
We will relate
to existing literature in our discussion. 

\color{black}
\subsection{Notation} Before analysing solutions to the CIFM, we discuss the notation
used in this section. We use
$| \blank |_{\infty}$ to denote the $\infty$-norm on $\CSet^m$.
We denote by $C(X,Y)$ the set of continuous
functions from $X$ to $Y$, and
use $C(X)$ when $Y = \RSet$. We denote by $B(X)$ ($BC(X)$) the set of
real-valued bounded (real-valued bounded, continuous) functions defined on $X$.
Further, for a positive number $\eta$, we shall use the
following exponentially weighted Banach spaces:
\[
  \begin{aligned}
    L^1_\eta(\RSet) & = 
  \Big\{ 
    u \colon \RSet \to \RSet \colon
    \Vert u \Vert_{L^1_\eta} = \int_\RSet e^{\eta x} |u(x)| \, dx < \infty
  \Big\}, \\
  C_\eta(\RSet,\CSet^m) & = 
  \Big\{ 
    u \in C(\RSet,\CSet^m) \colon 
    \Vert u \Vert_{C_{m,\eta}} = \sup_{x \in \RSet} e^{-\eta |x|} \, |u(x)|_{\infty}
    < \infty
  \Big\}.
\end{aligned}
\]
\color{black}

\subsection{Characterisation of solutions to the CIFM via the voltage mapping}
\label{subsec:TWCharacterisation}
We begin by discussing in what sense a voltage function $v$ satisfies the CIFM formal
evolution equation~\cref{eq:contMod}. While we eschew the definition of the CIFM as a
dynamical system on a Banach space (a characterisation that is currently unavailable
in the literature), we note that progress can be made for voltage profiles with a
constant and finite number of spikes for $t \in  \RSet$. This class of
solutions is sufficiently large to treat travelling waves, and small perturbations to
them. 

We make a few assumptions on the network coupling, and
we restrict the type of firing functions and solutions of interest, as follows:
\begin{hypothesis}[Coupling functions] \label{hyp:synapticFunctions}
  The connectivity kernel $w$ is an even function  in $C(\RSet) \cap L^1_\eta(\RSet)$,
  for some $\eta >0$. The post-synaptic function $\alpha \colon \RSet \to
  \RSet_{\geq 0}$ can be written as $\alpha(t) = p(t)H(t)$, where $H$ is the
  Heaviside function, and $p \colon \RSet_{\geq 0} \to \RSet$ is a bounded and
  everywhere differentiable Lipschitz function, hence $p,p' \in B(\RSet)$.
\end{hypothesis}

\begin{definition}[$m$-spike CIFM solution]\label{def:vM}
  Let $m \in \NSet$ and $I \in \RSet$. A function $v_m \colon
  \RSet^2 \to \RSet$ is an $m$-spike CIFM solution if there exists $\tau
  =(\tau_1,\ldots,\tau_m) \in C(\RSet,\RSet^m)$ such that
  $\tau_1 < \ldots < \tau_m$ on $\RSet$ and
  \begin{align}
  & \begin{aligned}
    v_m(x,t) = I & + \sum_{j\in \NSet_m} \int_{-\infty}^t \int_{-\infty}^\infty\!\!\! \exp(z-t) w(x-y)
	       \alpha(z-\tau_j(y)) \, dy \, dz \\
	       & - \sum_{j\in \NSet_m} \exp(\tau_j(x)-t) H(t-\tau_j(x)),
	       \qquad (x,t) \in \RSet^2 
  \end{aligned} \label{eq:vProfile}\\
  & v_m(x,t)=1, \qquad  (x,t)  \in \FSet_\tau,
  \label{eq:vCrossings} \\
  & v_m(x,t) < 1, \qquad (x,t) \in \RSet^2 \setminus \FSet_\tau,
  \label{eq:vBounded}
  \end{align}
  where 
  \[
  \FSet_\tau = \bigcup_{j \in \NSet_m} \{ (x,t) \in \RSet^2 \colon t = \tau_j(x)\}.
  \]
  We call $\tau$ and $\FSet_\tau$ the firing functions and the firing set of $v_m$,
  respectively.
\end{definition}
The definition above specifies how we interpret solutions to~\cref{eq:contMod}, and
is composed of three ingredients: (i) \Cref{eq:vProfile}, which derives from
integrating~\cref{eq:contMod} on $(-\infty,t)$, and expresses a mapping between the set
of $m$ firing functions $\tau$ and the voltage profile; (ii) System
\cref{eq:vCrossings}, which couples the firing functions by imposing the threshold
crossings; (iii) A further condition on $v_m$, ensuring that the solution has exactly
$m$ spikes, attained at the firing set; this is necessary because, as we shall see
below, it is possible to find a set of $m$ functions $\tau$ satisfying
\Crefrange{eq:vProfile}{eq:vCrossings}, but exhibiting a number of threshold
crossings greater than $m$.

We now aim to characterise $m$-spike CIFM solutions by means of a \textit{voltage
mapping}, which can be conveniently linearised around a firing set, and
is a key tool to construct waves and analyse their stability. Inspecting
\cref{eq:vProfile} we note that the voltage profile features two contributions, one
from the (synaptic) coupling functions $w$ and $\alpha$, and one from reset
conditions. This observation leads to the following definitions:
\color{black}
\begin{definition}[Synaptic, Reset, and Voltage mappings]\label{def:SR}
  Let $u: \RSet \to \RSet$. We define the synaptic
  operator, $S$, and the reset operator, $R$, by
  \begin{align} 
    (S u)(x,t)  & = \int_{-\infty}^t \int_{-\infty}^\infty \exp(z-t) w(x-y)
    \alpha(z-u(y)) \, dy \, dz, & (x,t) \in \RSet^2,  \label{eq:S} \\
    (R u)(x,t) & = -\exp(u(x)-t) H(t-u(x)),  & (x,t) \in \RSet^2. \label{eq:R}
  \end{align}
  Further, let $m \in \NSet$, $I \in \RSet$ and $\tau \in C(\RSet,\RSet^m)$. The
  $m$-spike voltage mapping, $V_m$, is the operator defined as
  \begin{equation}\label{eq:voltageMapping}
    V_m\tau = I + \sum_{j \in \NSet_m} ( S \tau_j + R \tau_j).
  \end{equation}

\end{definition}
These operators map univariate functions, such as a firing function, to bivariate
functions, such as the spatio-temporal voltage profile. Under
\cref{hyp:synapticFunctions} it holds $S \colon C(\RSet) \to BC(\RSet^2)$, $R \colon
C(\RSet) \to B(\RSet^2)$, hence $V_m \colon C(\RSet) \to BC(\RSet^2)$ (see \ref{supp:prop:SRMappings}).
\color{black}

By construction, the voltage operator characterises $m$-spike CIFM solutions, as the
following proposition shows.
\begin{proposition}\label{prop:voltageMapping}
  Let $m \in \NSet$, $I \in \RSet$. An $m$-spike CIFM solution exists if,
  and only if, there exists $\tau \in C(\RSet,\RSet^m)$ such that
  \begin{align}
    &V_m \tau = 1,  && \text{in $\FSet_\tau$}, 
      \label{eq:voltageMappingThresholds}\\
      &V_m \tau < 1,&& \text{in $\RSet^2 \setminus \FSet_\tau$}
  \end{align}
\end{proposition}
\begin{proof}
  The statement follows by setting $v_m(x,t) = (V_m \tau)(x,t)$ and applying the
  definition of the voltage mapping, \Cref{eq:voltageMapping}.
\end{proof}

%

\cref{prop:voltageMapping} implies that the voltage of an
$m$-spike solution can be computed for any $(x,t) \in \RSet^2$ once the
firing functions $\tau$ are known. The spatio-temporal profile of an $m$-spike
solution is determined entirely by its firing functions. This aspect, which
underlies the formal evolution equation \cref{eq:contMod} and the literature which
analyses it, is a key part of what follows and, as we shall see below, it also
suggests a natural way to compute travelling waves, and determine their linear
stability.
A first step in this direction is the definition of travelling waves via the
voltage mapping.

\subsection{Travelling waves with m-spikes (\tw{m})} Following \cref{prop:voltageMapping},
we can capture travelling waves with $m$ spikes (\tw{m}) using the voltage mapping,
and a set of parallel firing functions. Henceforth, we will assume without
loss of generality that the propagating speed of the wave is positive: for any wave
with $c>0$, there exists a wave with speed $-c$, and the wave profiles related by the
transformation $x \to -x$.



\begin{definition}[\tw{m}] Let $m \in \NSet$, $c>0$, and let $T \in \RSet^m$ with
  $T_1 < \cdots <T_m$. A travelling wave with
  $m$ spikes (\tw{m}), speed $c$, and coarse variables $(c,T)$
  is an $m$-spike CIFM solution with firing functions $\{ \tau_j(x) = x/c + T_j \}_{j \in
  \NSet_m}$. 
\end{definition}

To each travelling wave solution is associated a travelling wave profile which is
advected with propagation speed $c$. From \cref{prop:voltageMapping} we expect this
profile to be determined entirely by the firing functions, as confirmed in the
following result.

\begin{proposition}[\tw{m} profile]\label{prop:nu}
  A \tw{m} with speed $c$ satisfies $(V_m\tau)(x,t)= \nu_m(ct-x;c,T)$,
  and its $(c,T)$-dependent travelling wave profile $\nu_m$ is given by
  \begin{equation} \label{eq:nuXi}
  \begin{aligned}
    \nu_m(\xi;c,T) 
    = I  & - \sum_{j \in \NSet_m}
    \exp\bigg( -\frac{\xi -cT_j}{c} \bigg) H\bigg(\frac{\xi- c T_j}{c}\bigg) \\ 
    & + \frac{1}{c} \sum_{j \in \NSet_m}
    \int_{-\infty}^\xi \exp\bigg( \frac{z-\xi}{c} \bigg) \int_0^\infty
    w(y-z+cT_j) p(y/c) \, dy \,dz.
   \end{aligned}
   \end{equation}
\end{proposition}
\begin{proof} 
  See \cref{supp:proof:prop:nu}.
\end{proof}

\cref{prop:nu} shows that the travelling wave profile is completely determined by
the vector $(c,T) \in \RSet_{> 0} \times \RSet^m$, that is, $(c,T)$ is a vector of
coarse variables for the travelling wave. 
In the discrete model we introduced an auxiliary spatially-extended variable for the
model, the synaptic input $\{s_i(t)\}_i$ defined in \cref{eq:xsub}. In the
continuum model, the corresponding variable is the function $s_m(x,t) = \sum_{j \in
\NSet_m}(S\tau_j)(x,t)$, which in a \tw{m} satisfies $s_m(x,t) = \sigma_m(ct-x;c,T)$, with
\begin{equation}\label{eq:sigXi}
 \sigma_m(\xi; c,T) = \frac{1}{c} \sum_{j=1}^m \int_0^\infty w(y-\xi+cT_j)
  p(y/c) \, dy.
\end{equation}

\subsection{Travelling wave construction}
\begin{figure} \label{fig:TW5-TW20}
  \centering
  \includegraphics{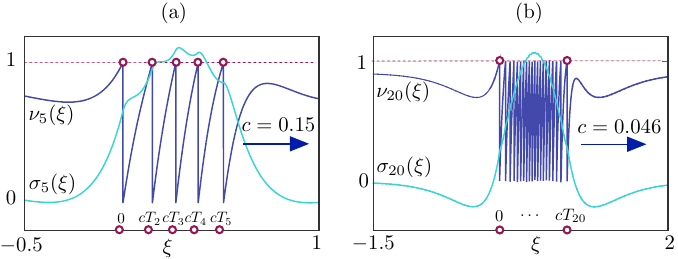}
  \caption{Wave profiles for (a) \tw{5} and (b) \tw{20} obtained by
  solving \cref{prob:TWm} for $m=5$ and $m=20$, respectively, and then subsitituting 
  $(c,T_{1},\dots,T_m)$ into the expression for voltage
  profile \cref{eq:nuXi} and synaptic profile \cref{eq:sigXi}. The profile $\nu$ is
  computable at any $\xi \in \RSet$, here we plot it using an arbitrary grid in the
  intervals (a) $[-0.5,1]$ and (b) $[-1.5,2]$. Parameters as in~\cref{tab:parameters}
  with (a) $\beta = 4.5$ and (b) $\beta = 7.7 $.} 
\end{figure}

\Cref{prop:nu} suggests a simple way to compute a \tw{m}, by
determining its $m+1$ \emph{coarse variables} $(c,T)$, as a solution to the following
\emph{coarse problem}:

\begin{problem}[Computation of \tw{m}]\label{prob:TWm}
  Find $(c,T) \in \RSet_{> 0} \times \RSet^{m}$ such that  $T_1< \cdots < T_m$ and
  \begin{align}
      T_1 & = 0,  \label{eq:phaseCond} \\
      \nu_m(cT_i^-; c, T)  & = 1, \qquad \text{for $i \in \NSet_m$}, 
	  \label{eq:nuThreshCross} \\
      \nu_m(\xi; c,T) & < 1, \qquad 
                        \text{on $\RSet \setminus \cup_{j \in \NSet_m}\{cT_j^-\}$}.
	\label{eq:lessThanOne}
   \end{align}
\end{problem}
\Cref{eq:nuThreshCross} of the coarse problem imposes that the travelling wave
profile crosses the threshold $1$ when $\xi \to cT^-_j$, which is a necessary and
sufficient condition to ensure $v_m = 1$ in $\mathbb{F}_\tau$ (see \cref{cor:limits}).  
As expected, if $\nu_m$ is a travelling wave profile, then so is
$\nu_m(\xi+\xi_0)$ for any $\xi_0 \in \RSet$; \Cref{eq:phaseCond} fixes the phase of
the travelling wave, by imposing that the profile crosses threshold as $\xi \to 0^-$.

If $m=1$, \Crefrange{eq:phaseCond}{eq:nuThreshCross} of the coarse problem
reduce to a compatibility condition for the speed $c$,
\[
c \int_{-\infty}^0 \int_0^\infty \exp(s) w\big(c(y-s)\big) p(y) \, dy \,ds =
I-1,
\]
which implicitly defines an existence curve for \tw{1} in the ($c$,$I$)-plane. This
result is in agreement with what was found in
\cite{Osan:2004ko,Ermentrout1998c}. Existence curves in other
parameters are also possible, and are at the core of the numerical bifurcation
analysis presented in detail in the sections below.

For $m>1$, the coarse problem must be solved numerically. A simple solution strategy
is to find a candidate solution using Newton's method for the system of $m+1$
transcendental equations \crefrange{eq:phaseCond}{eq:nuThreshCross}, with $\nu_m$
given by \cref{prop:nu}, and with initial guesses estimated from direct simulation of
the discrete model with large $n$, or from a previously computed coarse vector.
The candidate solution can then be evaluated at arbitrary $\xi \in
\RSet$, hence it is accepted if \cref{eq:nuThreshCross} holds on a spatial grid covering
$[-L,L] \subset \RSet$, with $L \gg 1$. In passing, we note that this procedure is
considerably cheaper than a
standard travelling wave computation for PDEs, which requires the solution of a
boundary value problem, and hence a discretisation of differential operators on
$\RSet$. Depending on the particular choice of $\alpha$ and $w$, the profile
$\nu_m$ is either written in closed form, as is the case for the
choices~\cref{eq:alphaAndW}, or approximated using standard quadrature rules. 

A concrete calculation is presented in \cref{fig:TW5-TW20}, where we show travelling
wave profiles and speeds of a \tw{5} and a \tw{20}. In passing, we note that the
synaptic profile of a \tw{m} at a given time is similar to a bump, but displays
modulations at the core (visible in \cref{fig:TW5-TW20}), as predicted by the
Heaviside switches in~\cref{eq:sigXi}. Travelling waves with a large number of
spikes, such as these ones, have not been accessible to date.

\begin{remark}
  \cref{fig:TW5-TW20} shows that profiles with $\nu_m(cT_j^-)=1$ propagate with
  \textit{positive} speed, and this does not contradict the numerical simulations in
  \cref{fig:exampleWaves}, where solutions profiles with $v_m(x,\tau_j(x)^-) = 1$
  propagate with \textit{negative} speed. This is a consequence of choosing $\xi = ct
  -x$ (as in~\cite{Osan:2004ko}), hence initial conditions for the time simulations
  are obtained by reflecting $\nu_m$ about the $y$ axis, since $v_m(x,0) = \nu_m(-x)$. 
\end{remark}


%
%

\section{Wave Stability}\label{sec:TWstability}

\begin{figure}
  \centering
  \includegraphics{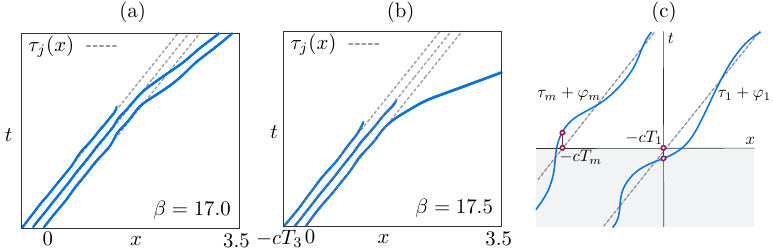}\label{fig:TildeTau}
  \caption{(a)-(b): Examples illustrating the destabilisation of a \tw{3} solution. A
    time simulation of the DIFM is initialised using wave profiles obtained solving
    \cref{prob:TWm} for $m = 3$ at (a) $\beta = 17$ and (b) $\beta = 17.5$.
  Parameters as in~\cref{tab:parameters}, domain half-width $L =4$ and network
  size $n =1000$. The firing functions $\{\tau_j\}$ are plotted for reference.
  Oscillatory perturbations to the firing functions do not decrease with time, hence
  the wave is unstable. The dynamics leads to stable (a) \tw{2} and (b)
  \tw{1} solutions. (c): Perturbations $\tau + \phi$ to the firing functions $\tau$ of a
\tw{m}. At $t=0$ each firing function $\tau_i$ is perturbed by an
amount $\phi_i(-cT_i)$. A \tw{m} is linearly stable if $\phi_i(-cT_i)$ being small
implies that $\phi_i(x)$ stays small for all $x \in (-cT_i,\infty)$ and $i \in
\NSet_m$ (see \cref{def:linearStability}).}
\end{figure}

The time simulations in \cref{sec:discreteModel} demonstrate that, for sufficiently
large values of $\beta$, travelling waves with a variable number of spikes coexist and
are stable. It is natural to ask whether these waves destabilise as $\beta$, or any
other control parameter of the model, is varied.
An example of a prototypical wave instability is presented in \cref{fig:TildeTau} for
\tw{3}: a travelling wave is computed solving \cref{prob:TWm}, and this solution is
used as initial condition for a DIFM simulation with $n=1000$ neurons. For
sufficiently large $\beta$, the wave is unstable,
as exemplified by the raster plots in \cref{fig:TildeTau}(a)--(b),
in that the firing functions never return to the ones of a \tw{3}. 
%
 
\cref{fig:TildeTau} shows that the firing set of the solution is composed of 3
disjoint curves, initially close to the ones of a \tw{3}, from which they depart
progressively. Ultimately, some firing functions terminate, and the dynamics displays
an attracting \tw{2} or \tw{1}. Capturing the transitions from a \tw{m} to
a travelling wave with fewer spikes is a nontrivial task. Studying the
\textit{nonlinear stability} is not possible with the current definition of CIFM
solutions, which require a constant number of spikes.
%
The voltage mapping, however, opens up the possibility of studying the \emph{linear
stability} of \tw{m}: the spatio-temporal voltage profile of an $m$-spike
solution is determined by its firing functions, $\tau$, via \cref{eq:voltageMapping};
small perturbations $\tau + \phi$ to $\tau$, induce small perturbations to
the spatio-temporal profile, and we expect that a suitable linearisation of the
voltage mapping carries information concerning the asymptotic behaviour of these
perturbations.
 \begin{figure} \label{fig:bif-diag-TW3}
  \centering
  \includegraphics{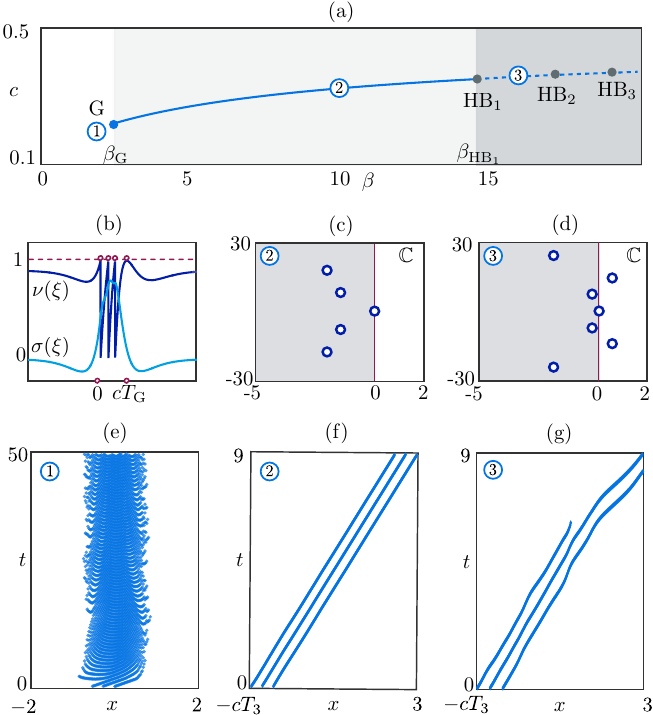}
  \caption{(a) Branch of $\tw{3}$ solutions in the parameter $\beta$, using $c$ as
    solution measure. The branch originates
    at a grazing point $G$, illustrated by the profile in (b). As
    $\beta$ increases, three pairs of complex conjugate roots of $E$ (see
    \Cref{eq:EDefinition}) cross the imaginary axis at the oscillatory (Hopf)
    bifurcation points $\hb1$, $\hb2$, $\hb3$. Panels (c)
    and (d) show selected roots of $E$, before and after $\hb1$, at $\beta = 10$ and
    $16$, respectively. (e)--(g) Raster plots for time simulations of the
    DIFM with $n=500$ and domain half-width $L=3$, initialised from solutions to
    \cref{prob:TWm} at $\beta = 2.17$, $10$, and $16$, respectively. The simulations
    show the dynamics of the model for $\beta < \beta_\textrm{G}$ (where a \tw{3}
    does not exist in the continuum limit), for $\beta \in
    (\beta_\textrm{G},\beta_{\textrm{HB}_1})$ (where \tw{3} is stable according to
    the analysis in (c)), and for $\beta > \beta_{\textrm{HB}_1}$ (where \tw{3} is unstable to
    oscillatory perturbations, as predicted in (d)). Parameters as in
    \cref{tab:parameters}, with $d_1=0$.}
  \end{figure}
 
\color{black}
Building on the definitions and results in \cref{sec:TWm}, we have formalised the concept of linear stability, and developed an
algorithm for \tw{m} linear stability computations. We give here a nontechnical
summary of the main results, and we refer to the Supplementary Material for a longer
discussion including definitions, theorem statements, and proofs.

\textit{Result 1 (\cref{lemma:linearVM}).} If two distinct $m$-spike solutions have firing functions $\tau$
and $\tau+\phi$ then, to leading order, $\phi$ is in the kernel
of a bounded linear operator, $L \colon C_\eta(\RSet,\RSet^m) \to
C_\eta(\RSet,\RSet^m)$,
obtained by linearising the voltage mapping $V_m$ around $\tau$. We recall that
$\eta$ bounds the decay rate of the connectivity function, $w \in L^1_\eta(\RSet)$
(see \cref{hyp:synapticFunctions}). This implies that admissible perturbations $\phi$
are allowed to grow exponentially
as $|x| \to \infty$, at a rate at most equal to the decay rate of $w$.

\textit{Result 2 (\cref{def:linearStability} and surrounding discussion).} As
for \tw{m} existence, linear stability is characterised via firing functions: loosely
speaking, a wave
with firing functions $\tau$ is linearly stable to perturbations $\phi \in \ker L$ if
the firing sets $\FSet_\tau$ and $\FSet_{\tau+\phi}$ are close around $t=0$, and
remain close for all positive times (see also caption to \cref{fig:TildeTau}(c)).

\textit{Result 3 (\cref{lem:twStability} and following discussion).} Linear stability is determined by a complex-valued function $E \colon \DSet_{-\eta,\eta} \to \CSet$, where
$\DSet_{-\eta,\eta} = \{ z \in \CSet \colon -\eta \leq \real z \leq \eta \}$. A
\tw{m} is stable to perturbations of the type $\phi(x) = \Phi e^{\lambda x} + \Phi^*
e^{\lambda^* x}$ (where $\Phi \in \RSet^m$ and the star denotes complex
conjugation) if all nonzero roots $\lambda$ of $E$ have strictly negative real
parts. The function $E$ can be
evaluated using the coarse wave variables $(c,T)$.
\color{black}

\section{Bifurcation structure of travelling waves}\label{sec:bif-structure-TW}

  The pseudo-arclength continuation routines developed in
  \cite{rankin2014continuation,avitabile2020zenodo}
  have been used to compute solutions to \cref{prob:TWm}, continue waves in parameter
  space, and investigate their stability. A \tw{m} is constructed by
  solving \cref{prob:TWm} in the coarse variables $(c,T) \in \RSet_{>0} \times
  \RSet^{m}$, which is sufficient to reconstruct the wave profile~\cref{eq:nuXi}, and
  the corresponding synaptic profile~\cref{eq:sigXi}; 
in addition, starting from a solution to \cref{prob:TWm}, the linear asymptotic
stability of a \tw{m} is determined by finding roots of the $(c,T)$-dependent
nonlinear function $E$ defined in \cref{eq:EDefinition}.

\cref{fig:bif-diag-TW3} shows the bifurcation
structure of \tw{3}, which is common to most travelling waves found in the model.
The simulations in \Cref{sec:discreteModel} suggest to take the synaptic timescale
parameter $\beta$ as the principal continuation parameter. We use the wavespeed $c$
as solution measure. A branch of solutions originates from a grazing point (G, see
below for a more detailed explanation) and it is initially stable, before
destabilising at a sequence of oscillatory bifurcations
($\hb{1}$--$\hb{3}$), as seen in \cref{fig:bif-diag-TW3}(a). In passing, we note that
there exists a second, fully unstable, branch of \tw{3} solutions characterised by a
slower speed and a smaller width.
This branch, which we omit from the bifurcation diagrams for simplicity, also originates at a
grazing point. 

\subsection{Grazing points} In a wide region of parameter space, branches of
\tw{m} solutions originate at a grazing point $\beta = \beta_\textrm{G}$, as
seen in \cref{fig:bif-diag-TW3}(a)--(b) for \tw{3}\footnote{Note that
$\beta_\textrm{G}$ depend on $m$, but we omit this dependence to simplify notation.
The same is true for other quantities in the paper such as $c$ and $T_\textrm{G}$,
for instance.}. At a
grazing point the
\tw{m} profile crosses threshold $m$ times, and attains the threshold
tangentially at a further spatial location, $cT_\textrm{G}$, as shown in
\cref{fig:bif-diag-TW3}(b). This tangency exists at the critical value $\beta =
\beta_\textrm{G}$, signalling a non-smooth transition and a branch termination.
For $\beta > \beta_\textrm{G}$ we observe profiles with exactly $m$ threshold
crossings (a branch of \tw{m} solutions). These profiles exhibit a further local
maximum, which is strictly less than $1$ by construction, at a point
$\xi_\textrm{max} > cT_m$. As $\beta
\to \beta^+_\textrm{G}$, we observe $\xi_\textrm{max} \to cT_\textrm{G}^+$ and
$\nu(\xi_\textrm{max}) \to 1^-$, until the threshold is reached at $\beta =
\beta_\textrm{G}$, where the tangency originates.

For $\beta < \beta_\textrm{G}$, we find solutions to the nonlinear problem
\crefrange{eq:phaseCond}{eq:nuThreshCross} for which $V_m \tau > 1$ in a bounded
interval of $\RSet$. Since these states violate the condition \cref{eq:lessThanOne},
they do not correspond to $\tw{m}$ solutions, and we disregard them (the branch
terminates at $\beta_G$). We note, however, that in a neighbourhood of
$\beta_\textrm{G}$ there exist branches of travelling wave solutions with different
number of threshold crossings (as it will be shown below).
 
We found grazing points for every $\tw m$ with $2 \leq m \leq 230$, for the parameters in
\cref{tab:parameters} with $d_1=0$. We observe that for $\beta < \beta_\textrm{G}$
the system evolves towards a DIFM bump attractor (see \cref{fig:bif-diag-TW3}(e)).
Understanding the origin of this transition is the subject of the following sections.
 
Grazing points are found generically as a secondary control parameter is varied, and
$2$-parameter continuations of grazing points can be obtained numerically, by freeing
one parameter and imposing tangency of the wave profile at one additional point (see
\cref{prob:G} in \cref{sec:supp:TwoParameterContinuation}).

\subsection{Oscillatory bifurcations}
Along the \tw{m} branch, we compute and monitor the roots of $E$ with the largest
real part. \cref{fig:bif-diag-TW3}(c)-(d) show examples
for $\tw{3}$ at $\beta =10$ and $\beta = 16$ respectively. At
$\beta = 10$, we observe a root at $0$, as expected, and other roots with small
negative real part: the wave is therefore linearly asymptotically stable to
firing-threshold perturbations $x \mapsto \Phi e^{\lambda x} + \Phi^* e^{\lambda^*
x}$, with $E(\lambda) = 0$ and $\Phi \in \ker[D - M(\lambda)]$
\textcolor{black}{(see \cref{lem:twStability})}, as confirmed
via simulation in \cref{fig:bif-diag-TW3}(f). In contrast, there exists a pair of
unstable complex conjugate roots for the solution at $\beta = 16$,
indicating an oscillatory (Hopf) instability, which is also confirmed by direct
simulation, in \cref{fig:bif-diag-TW3}(g): after the initial oscillatory instability,
the system destabilises to a \tw{2}. It should be noted that, in other
regions of parameter space and for simulations with different network sizes, we
observed a \tw{3} destabilise to a \tw{1} or the homogeneous steady state. 

We expect that branches of periodically modulated \tw{m} solutions
(which are also supported by neural fields~\cite{Ermentrout:2014bw,Coombes:2014uy})
emerge from each of the Hopf bifurcations reported in \cref{fig:bif-diag-TW3}(a). 
\color{black}
We
note that we could not find stable structures of this type via direct simulations
near the onset of the instability, indicating that the Hopf bifurcations may be
subcritical.
\color{black}
While it
is possible to extend our framework to continue such periodic states, we did not
pursue this strategy here. 

As shown in \cref{fig:bif-diag-TW3}(a), the \tw{3} branch undergoes a sequence of
Hopf bifurcations $\{ \hb{i} \}_i$: our stability analysis shows several
pairs of complex conjugate roots progressively crossing the imaginary axis as
$\beta$ increases: the computation in \cref{fig:bif-diag-TW3}(d), for instance, is for a
solution at $\beta \in (\beta_{\textrm{HB}_1}, \beta_{\textrm{HB}_2})$. We have
verified numerically (not shown) that the firing functions of spatio-temporal DIFM
solutions in this region of parameter behave as predicted by the leading eigenvalues
in \cref{fig:bif-diag-TW3}(d), that is, they feature two dominant oscillatory modes:
one stable, and one unstable.
Similarly to grazing points, 
Hopf bifurcations can be continued in a secondary parameter (see \cref{prob:Hopf}
in~\cref{sec:supp:TwoParameterContinuation}).

\subsection{Nested branches of travelling waves} 
\begin{figure} \label{fig:c-beta-TW1-TW160}
  \centering
  \includegraphics{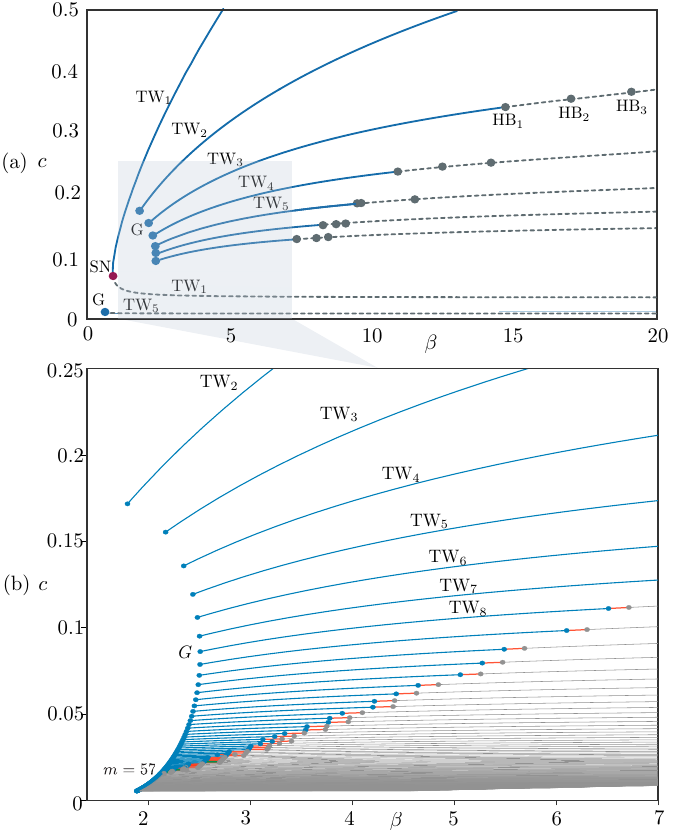}
  \caption{Bifurcation structure of $\tw{m}$ branches for $m=1,\ldots,160$ in the
    parameter $\beta$. (a): For $m \geq 3$, branches are similar to the one shown in
  \cref{fig:bif-diag-TW3}(a). As $m$ increases, the waves become slower and their
  stability region narrower. The shaded area in (a) is enlarged in (b): the
  inset shows selected branches for $m=2,\ldots,160$; oscillatory instabilities occur
  within the red segments \color{black}(connecting a stable solution in blue, to an
  unstable solution in grey),\color{black} and the branches with $m \geq 57$
  are fully unstable (solid grey lines). We used here the same data as in
  \cref{fig:bumpSummary}, but we present it in terms of $c$, not $\Delta$. Parameters
  as in \cref{tab:parameters}, with $d_1 = 0$.} 
\end{figure}

We computed branches of \tw{m} solutions for increasing values of $m$, as reported in
\cref{fig:c-beta-TW1-TW160}(a), using DIFM simulations as initial guesses. In
\cref{fig:bumpSummary} waves were represented by their width, whereas here we use the
propagation speed $c$.
In the
region of parameter space explored in the DIFM model, branches with $m\geq 2$ feature
a grazing point for low
$\beta$, and branches with $m \geq 3$ display sequences of Hopf Bifurcations,
following the scenario already discussed in \cref{fig:bif-diag-TW3}(a). 
In this region, the \tw{1} branch has a distinct behaviour, featuring a saddle node
point in place of a grazing point. For each \tw{m} branch terminating at a grazing point,
there is a corresponding slow unstable branch originating at a different
grazing point: in \cref{fig:c-beta-TW1-TW160}(a) this behaviour is exemplified
by plotting the fully unstable slow $\tw{5}$ branch (the branch with slowest waves in the
figure), but is omitted for all other branches. The two $\tw{5}$ branches should be
understood as a ``broken saddle-node".
The bifurcation structure of \cref{fig:c-beta-TW1-TW160}(a), valid for the CIFM,
supports numerical simulations of the DIFM, in which a $\tw{m}$ destabilises at
\hb{1}, and gives rise to a new travelling wave
state, \tw{m'} with $m'<m$ (see for instance \cref{fig:TildeTau,fig:bif-diag-TW3}). 

These \emph{coexisting \tw{m} branches} are nested in a characteristic fashion, so
far unreported in the literature; the
higher $m$, the slower the wave, and the narrower the stable interval
between $G$ and \hb{1}. This structure is noteworthy: firstly, it is
known that the speed of \tw{1} typically changes \emph{as a secondary parameter is
varied}~\cite{Ermentrout1998c,Bressloff:1999a,Bressloff:2000dq}; however, in networks
with purely excitatory kernels, waves with multiple threshold crossings coexist, and
their speed does not depend strongly on $m$~\cite{Ermentrout1998c}, which has been
a principle reason for studying approximately and analytically the only tractable case,
$m=1$~\cite[Section 5.4]{bressloff2014waves} (this scenario is also confirmed by our
calculations, see~\cref{fig:sup:purelyexcitatory_bif_profiles}); secondly, it is
known that Hopf instabilities with purely excitatory connectivity kernel are possible
only if delays are present in the network~\cite{Bressloff:2000dq}.

The results in \cref{fig:c-beta-TW1-TW160} have been obtained using a methodology
that works for arbitrary $m$, and on generic connectivity kernels. They show that, when
inhibition is present: (i) coexisting nested branches of \tw{m} exist; (ii) the speed
of such waves depends
strongly on $m$, and in particular it is possible to construct waves with arbitrarily
small speed, by increasing the number of spikes; (iii) oscillatory instabilities are
present in models without delays, for sufficiently large $m$ and/or sufficiently large
$\beta$. As we shall see, the latter aspect plays a role in understanding the so
called \emph{bump attractor}.

\begin{figure}\label{fig:graze_m_figs}
  \centering
  \includegraphics{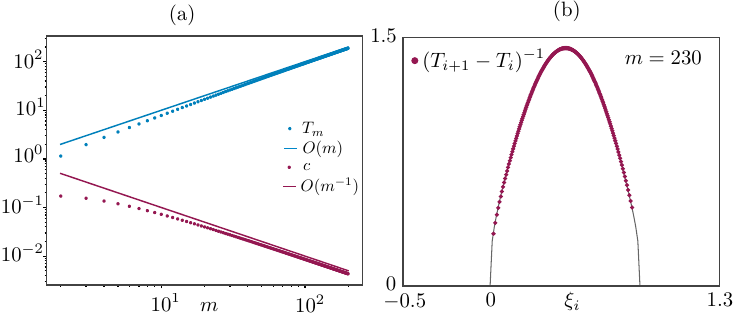}
  \color{black}
  \caption{(a) The quantities $c$ and $T_m$, evaluated at the grazing points $\beta =
    \beta_G$, are $O(m^{-1})$ and $O(m)$,
    respectively. Since $T_1 = 0$ for all waves, the quantity $cT_m$ measures the
    wave width, and we expect the sequence $\{cT_m\}_{m \in \NSet}$, the sequence of
    wave widths, to converge to a fixed value as $m \to \infty$. (b) The solid grey
    line is the spatial firing rate profile proposed in \cite{Laing:2001fc} for
    a non-wandering bump, red dots mark the
    instantaneous firing rate for \tw{230} at the grazing point, computed according
  to the formula $(T_{i+1}-T_i)^{-1}$ at position $x=cT_i$.}
  \color{black}
\end{figure}

\color{black}
\subsection{The bump attractor} 
From the grazing point of \tw{m}, one can compute the grazing point of \tw{m+1}.
For instance, from the \tw{3} grazing profile in \cref{fig:bif-diag-TW3}(b), we
obtain $(c,T_1,T_2,T_3,T_G)$. A grazing point can then be computed
solving~\cref{prob:G}, and its solution can be used to produce an initial
guess $(c,T_1,T_2,T_3, (T_3 + T_G) / 2,T_G)$ for a grazing point of \tw{4}.
Exploiting this iterative strategy, we compute grazing points and branches for
large values of $m$, obtaining the diagram in \cref{fig:c-beta-TW1-TW160}(b),
corresponding to the shaded area in \cref{fig:c-beta-TW1-TW160}(a). 

The branches accumulate as $m$ increases, and for $m \geq 57$, they are fully unstable
for this parameter set. The diagrams provide evidence that there exist unstable
waves with arbitrarily many spikes (i.e., with arbitrarily large $m$) and vanishingly small
speed. It seems therefore natural to postulate a relationship between these waves and
the bump structures found by Laing and Chow~\cite{Laing:2001fc} (see also
\cref{fig:bumpSummary,fig:exampleBumps,fig:bif-diag-TW3}(e)). 

\subsubsection{Spatial profile in non-wandering bumps} 
%
In the CIFM, we inspected travelling wave profiles for \tw{m} solutions at each of
the grazing points where they originate. The leftmost spike of each wave occurs at
$\xi_1=0$ by construction (see \cref{prob:TWm}), while its righmost spike is at
$\xi_m = cT_m$, which is therefore a proxy for the wave's width\footnote{Recall that
$c$ is also a function of $m$, but we omit this dependence for ease of notation.}.
\cref{fig:graze_m_figs}(a) shows $c$ and $T_m$, computed at the grazing points, as functions of
$m$: we find $c = O(m^{-1})$ and $T_m = O(m)$, therefore, we expect the sequence
$\lbrace\xi_m \rbrace_{m \in \NSet}$ to converge to a finite value $\xi_*$ as $m \to
\infty$. 

These data indicate that, as the wavespeed tends to zero, the growing number of spikes
are distributed in a fixed interval $[0,\xi_*]$. Hence, even though there exists no
stationary and spatially heterogeneous CIFM solution for finite $m$ (this
possibility is ruled out by \cref{def:vM}), there is evidence that an $m\to \infty$
limit of \tw{m} solutions exists, has $0$ speed, and displays a spatially
heterogeneous profile, localised in the region $x \in [0,\xi_*]$. Thus, the limiting
state possesses features of the \textit{stationary bumps} that are typically analysed
in continuum neural field models.

\begin{figure}\label{fig:bumpAttractor}
  \centering
  \includegraphics{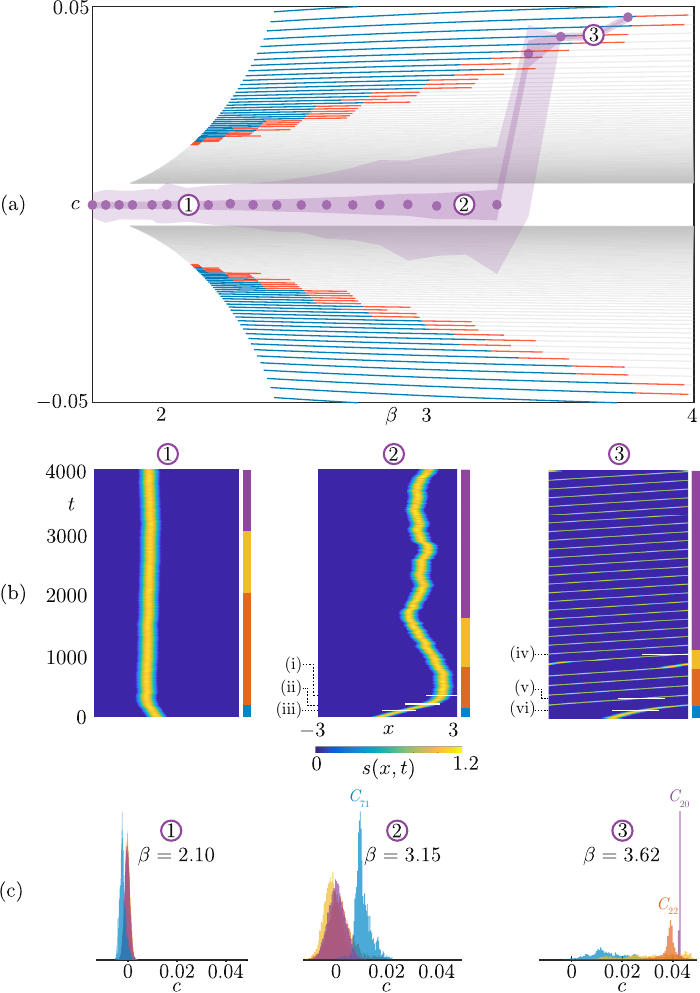}
  \caption{(a): Mean instantaneous speed ($\bar c$ in \cref{eq:statistics}, purple dots) and interval
    estimators ($[\bar c - \sigma_c, \bar c + \sigma_c]$ and
    $[c_\textrm{min},c_\textrm{max}]$, dark and light purple shades, respectively)
    in direct simulations of the DIFM,
    superimposed on \tw{m}
    branches of the CIFM (an inset of \cref{fig:c-beta-TW1-TW160}(b), which has been
    reflected about the $c=0$ axis to signpost waves with negative speed). The bump
    attractor is characterised by $\bar c \approx 0$, and fluctuations in speed
    that grow with $\beta$. (b): Exemplary solutions in (a) displaying an initial advection,
    followed by a bump attractor (1,2) or a stable wave (3). 
  \color{black}
    Snapshots of the
    rastergrams around times (i)--(vi) (white bars in the contour plots) are visible in \cref{fig:waveBumpComp}. (c)
    Histograms of the solution's instantaneous speed, computed in selected time
    intervals, indicated by blue, orange, yellow, and purple bars in (b). Sharp peaks
    indicate proximity of the orbit to a travelling wave, whose speed is indicated on top
    of the peaks ($C_{71}$, $C_{22}$, and
  $C_{20}$ for \tw{71}, \tw{22}, and \tw{20}, respectively), see also
\cref{fig:waveBumpComp}.}
  \color{black}
\end{figure}

To further substantiate this claim, we compare data of the slowest computed wave
(\tw{230} at the grazing point) to data of a non-wandering bump in the DIFM. The
DIFM also does not admit stationary spatially heterogeneous solutions,
but supports non-wandering bump attractors (see \cref{fig:exampleBumps}(a) and
\cref{fig:bif-diag-TW3}(e) for examples). In
such states, the dynamic is not stationary, with many asynchronous firing events
occurring at microscopic level; Laing and Chow noted that this state has
a spatially-dependent firing rate, for which they provide a closed-form
expression. They also showed that their analytical prediction is in agreement with
DIFM simluations of a non-wandering bump
attractor; the firing rate profile is therefore a macroscopic observable of a
non-wandering bump. 

\Cref{fig:graze_m_figs}(b) compares Laing and Chow's firing rate profile to the inverse
inter-spike time $1/(T_{i+1}-T_i)$ in the computed \tw{230}, that is, a proxy for
the firing rate at $x = \xi_i$. The agreement is excellent, confirming that,
from a macroscopic viewpoint, the DIFM bump attractors bear a strong relation to
\tw{m} solutions in the limit of large $m$. 

\subsubsection{Macroscopic observables of wandering bumps} 
We further investigate the bump attractor state in relation to the \tw{m}, away from
the non-wandering limit studied above: the
analysis of the CIFM, in the region of parameter space where the bump
attractor is observed, predicts the coexistence of the trivial
attracting solution $v(x,t) \equiv I$, with arbitrarily slow, unstable waves whose
spatial profile approximates that of a bump. Following the turbulence analogy, we
provide evidence that transient states to the DIFM bump attractor, or the bump
attractor itself, display features of the underlying unstable \tw{m}. We discuss
data for three travelling wave observables: instantaneous speed, instantaneous width,
and firing sets. 

\textbf{Instantaneous speed and width.}
We simulate the DIFM with $n=5,000$, initialising the model from an unstable
travelling wave of the CIFM, \tw{105}, and
estimate the instantaneous speed $c(t)$ of the numerical DIFM solution at $q$
time points $\{t_k \colon k\in \NSet_q\}$, using a level set of the synaptic
profile and finite differences, as follows:
\[
  z(t) = \max \{ x \in \SSet : s(x,t) = 0.1 \}, \qquad
  c_k = (z(t_k) -  z(t_{k-1}))/(t_k - t_{k-1}),
  \qquad
  k \in \NSet_q.
\]
A CIFM travelling wave solution corresponds to a constant $c$: when the DIFM solution
displays a wave for large $n$, 
the sequence $\{c_k\}_k$ converges to a constant value, if one disregards small
oscillations due to the finite $n$, and which vanish as $n \to \infty$. On the other
hand, we expect that no differentiable function $c(t)$ exists for a bump attractor.
However, useful information may be found in the mean, $\overline{c}$, standard deviation, $\sigma_c$, and
extrema, $c_\textrm{min}$, $c_\textrm{max}$, of the \emph{deterministic} scalar $c_k$
\begin{equation}\label{eq:statistics}
  \bar c  = \frac{1}{q} \sum_{k \in \NSet_q} c_k,
  \quad
  \sigma^2_c = \frac{1}{q-1} \sum_{k \in \NSet_q} (c_k-\bar c )^2,
  \quad
  c_\textrm{min} = \min_{k \in \NSet_q} c_k,
  \quad
  c_\textrm{max} = \max_{k \in \NSet_q} c_k.
\end{equation}
These quantities are computed for long simulations ($10,000$ time units)
after an initial transient ($1,000$ time units)
for various values of
$\beta$, and superimposed on the bifurcation
diagram of the CIFM model, in \cref{fig:bumpAttractor}(a): we plot $\bar c$ (purple
dots) and two interval estimators, $[\bar c - \sigma_c, \bar c + \sigma_c]$ (dark
purple shade) and $[c_\textrm{min}, c_\textrm{max}]$ (light purple shade). We recall
that the CFIM admits branches of waves with positive and negative speed, both plotted
in the figure, and that we omit slow unstable waves such as the one in
\cref{fig:c-beta-TW1-TW160}(a). Further, we conjectured above that branches of
unstable waves also exist in the white band around $c=0$.

\Cref{fig:bumpAttractor} shows that the bump attractor dynamics with respect to the variable
$c(t)$ is confined to a region where unstable \tw{m} solutions exist for low and
medium values of $\beta$. A similar behaviour is found for the instantaneous
bump widths, $\Delta(t)$, which can also be estimated from $z(t)$. The macroscopic variable
$\Delta(t)$ does not have large variations \textcolor{black}{within} a bump attractor. As shown in
\cref{fig:bumpSummary}, the average of $\Delta(t)$ for a wandering bump attractor is
located in the region of the bifurcation diagram where unstable \tw{m} are found.
\begin{figure}\label{fig:waveBumpComp}
  \centering
  \includegraphics[width=\textwidth]{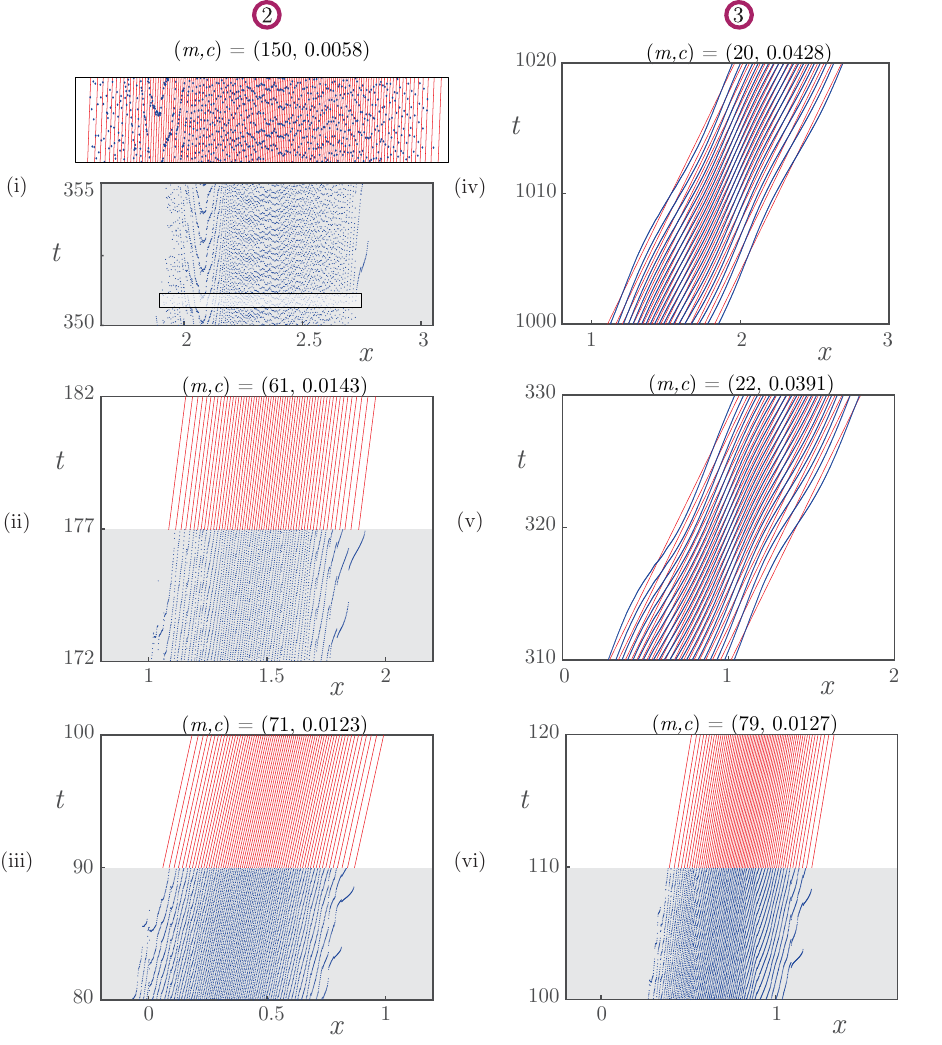}
  \color{black}
  \caption{Firing set of the DIFM solutions (blue dots) and of selected CIFM waves
    (overlayed red lines, with corresponding values of $m$ and $c$) near the times
    marked with a white tick (i)--(vi) in ~\cref{fig:waveBumpComp}(b). The
    firing set (vi) of orbit (3) has a recognisable travelling core which
    progressively loses firing functions at the edges, until it visits the weakly
    unstable \tw{22} and is attracted to the stable \tw{20}. The same initial
    condition with a different $\beta$ value leads to orbit (2). The firing sets
    (ii) and (iii) are qualitatively similar to (vi). The chaotic bump attractor
    (i) has distinctive travelling firing sets at the edges, visible in the grey raster
    plot: firing lines are lost to the right, and new travelling lines are injected
    into the core from the left, through a repeating V-shaped pattern.}
  \color{black}
\end{figure}

For low and medium $\beta$ values, we observe non-wandering and wandering bump
attractors, albeit the fine details of the dynamics depend on initial conditions.
\Cref{fig:bumpAttractor}(b)
shows 3 examples whose estimated average speeds appear also in \cref{fig:bumpAttractor}(a).
The space-time plots display an initial advection, followed by a bump attractor, or a
stable travelling wave. To gain insight into these transitions, we compute histograms of $c_k$
in selected time intervals, indicated by blue, orange, yellow, and purple bars in
\cref{fig:bumpAttractor}(b). Histograms that are sharply peaked around a nonzero
value provide evidence that the solution spends time close to a wave.
For instance, the purple histogram in \cref{fig:bumpAttractor}(c), orbit 3, has
been computed on a long time interval signposted with a purple bar on the right
vertical axis of \cref{fig:bumpAttractor}(b), orbit 3. The colormap of $s(x,t)$
in \cref{fig:bumpAttractor}(b) shows that orbit 3 approaches a stable
travelling wave, and the corresponding purple histogram is indeed close to a
Dirac delta centred at $C_{20}$, the speed of the stable \tw{20}. 

Before settling to \tw{20} the orbit spends time (orange bar in
\cref{fig:bumpAttractor}(b), orbit 3) near the unstable \tw{22}: there is a clear
transition in \cref{fig:bumpAttractor}(b), orbit 3 (after the orange bar), and the
corresponding orange histogram has a tail, but is sharply peaked around $C_{22}$.
This is in line with
with the observation that $c(t)$ has
growing oscillations around $C_{22}$, and indeed \tw{22} is
unstable. Similar considerations apply to \cref{fig:bumpAttractor}(b),
orbit 2, which visits the unstable \tw{71}.
 
\textbf{Firing sets.} In addition to speed, we compare the firing sets of solutions
labelled 2 and 3 in \cref{fig:bumpAttractor} to the ones of selected \tw{m}. The
former are transient solutions, the latter are invariant, and we overlay them in
\cref{fig:waveBumpComp}. The firing set of solution
3 around the time labelled (iv) \cref{fig:bumpAttractor}(b) is
visible in \cref{fig:waveBumpComp}(iv). From the initial condition at \tw{105},
propagating with positive speed, the solution slows down and ``sheds" firing
functions to the right of the profile, while the travelling firing set at the core
persists to oscillatory perturbations. For a visual comparison with CIFM waves, we
overlay in \cref{fig:waveBumpComp}(iv) a \tw{71} solution with a propagation
speed close to the transient. After this strongly nonlinear transient, the solution visits
the weakly unstable \tw{22}: in this transient, the firing set of the
DIFM solution clearly displays the oscillations predicted by the linear stability
theory for \tw{22} (see \cref{fig:waveBumpComp}(v)), before losing 2
further firing curves and being attracted to the stable \tw{20}
(see \cref{fig:waveBumpComp}(iv) and the purple, sharply peaked histogram in
\cref{fig:bumpAttractor}(c), label 3).

Solutions 2 and 3 in \cref{fig:bumpAttractor}(b) both start from \tw{105}, and the
latter displays a similar transient dynamics to the former, with a travelling core
and progressive loss of firing functions (\cref{fig:waveBumpComp}(ii)--(iii)),
accompanied by an increase in propagation speed. The bump attractor alternates phases
with small negative and positive propagation speed, as in \cref{fig:waveBumpComp}. As
expected, it is challenging to single out a matching wave in this highly chaotic
regime, albeit we present a comparison with \tw{150}. The bump still features
distinctive travelling firing sets at the edges, visible in the grey raster plot. The
right edge has a marked alignment of firing events, and some firing curves terminate
as in the other figures. Meanwhile, new firing curves are injected into the core from the
left, through a characteristic, repeated V-shaped pattern. When the bump attractor
propagates slowly with negative speeds, the V-shaped patterns are on the right, and
firing lines are shed on the left (not shown).

\color{black}
\subsection{Composite waves}\label{subsec:compositeWaves}
In addition to the waves studied thus far, we found by direct simulation waves whose
firing functions are split into well-separated groups, that is, firing functions in
the same group are closer to each another than they are to those in other
groups, see~\cref{fig:composite_formation_solution}. We call these structures
\textit{composite waves}, as they may be
formed via the interaction of travelling waves with various numbers of spikes. As in
other non-smooth dynamical systems~\cite{granados2017period}, we expect that
these solutions have discontinuities that are rearranged with respect to a \tw{m}. 

For illustrative purposes, we denote a composite wave with
$k\in\mathbb{N}$ groups by \tw{m_1} + \dots +
\tw{m_k}, where $\{m_i\}_{i=1}^k$ is a sequence of positive
integers specifying the number of spikes in each group. There are constraints
for the groups, dictated by dynamical considerations: for instance a \tw{1} +
\tw{3} cannot exist, because a \tw{1}, taken in isolation, is faster than a \tw{3}.
The construction of asymptotic profiles and computation of linear stability for
composite waves follow in the same way as defined in \cref{sec:TWm} and
\cref{sec:TWstability}.

In \cref{fig:composite_formation_solution}(a), we show a selection of of
composite waves near the \tw{3} branch. Roughly speaking, the wave profile along
each depicted branch comprises a \tw{3} as its leading group, followed by two
additional spike group that collectively form a compound satisfying the
travelling wave conditions (e.g.,~branch 1 combines a \tw{3}, a \tw{2} and a
\tw{1}).  The branches of composite waves are separate from each other and from
the previously computed \tw{m} branches in \cref{fig:c-beta-TW1-TW160}, however,
all branches possess a bifurcation structure similar to the one of the \tw{m}
discussed in the past section.  Moreover, we see that the magnitude of the speed
of the composite wave is bounded above by the magnitude of the speed of the
group at the leading edge of the wave (the slowest wave, \tw{3} in this case).  

%

\begin{figure}
\label{fig:composite_formation_solution}
  \centering
  \includegraphics{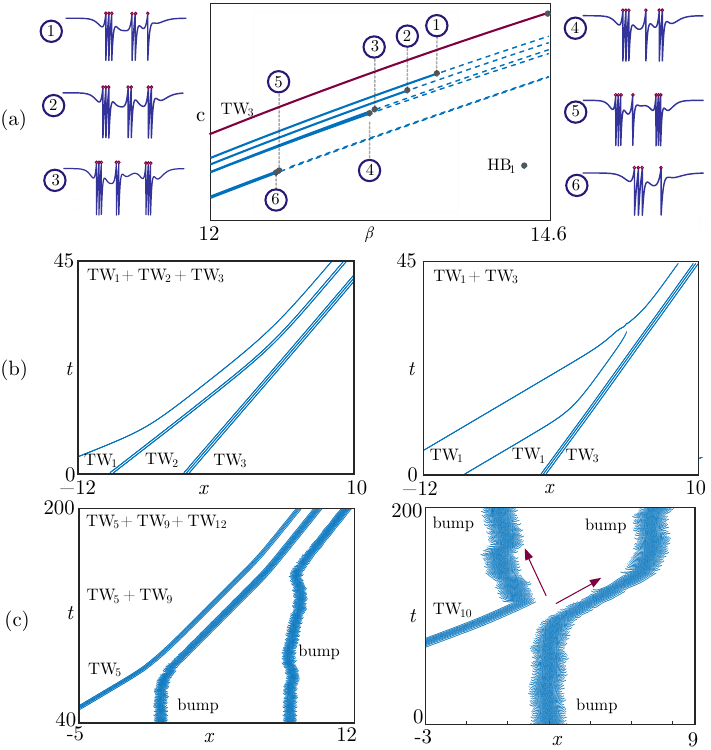}
  \caption{(a) Bifurcation diagram of selected composite waves. The red curve is a \tw{3}
    branch, as computed in \cref{fig:bif-diag-TW3}. The blue curves are branches of
    composite waves, featuring an approximate \tw{3} at the front of the wave. The
    composite waves are slightly slower than \tw{3}. The diagram shows selected
    profiles at the first oscillatory bifurcation points. 
    (b) Examples of composite waves obtained via collisions of multi-spike waves.
    (c) Collisions between $m$-spike propagating structures and wandering bumps
    generate composite waves (left) or bump repulsion (right), depending on initial
    conditions. Simulations in panels (b)--(c) have a lattice spacing of $\Delta x =
  2L/n = 0.01$.}
\end{figure}

Direct numerical simulation highlights that composite waves can be formed from the
interaction of multi-spike waves as shown in the left panel
of \cref{fig:composite_formation_solution}(b). Here we choose an initial
condition with well separated $\tw{1}$, $\tw{2}$ and $\tw{3}$ profiles.
Initially, these separated structures travel with different speeds (\tw{1} being
the fastest and \tw{3} the slowest, in line with what was found in
\cref{fig:c-beta-TW1-TW160}(a)).  After a transient, the waves come closer and
form a compound (the composite wave), with a common intermediate speed.  The
dynamics of composite waves depend greatly on the initial conditions: in the
right panel of \cref{fig:composite_formation_solution}(b), we see that an
initial condition in which a \tw{1} lies between another \tw{1} and a \tw{3}
leads to the extinction of the intermediate wave resulting in a composite wave
with a total of $4$ spikes.

Composite waves can also result from the collision between waves and wandering bumps
(\cref{fig:composite_formation_solution}(c), left panel). Here, we see a
transition of two bump states into a composite wave that is compounded with a
pre-existing \tw{5}. The interaction with the \tw{5} causes the left-most bump
to visit the branches of travelling wave solutions whereupon the combined state
settles on a stable \tw{5} + \tw{9}. This process is repeated for the
right-most bump, giving rise to an overall \tw{5} + \tw{9} + \tw{12}.
In the right panel of the
\cref{fig:composite_formation_solution}(c), we see that the same kind of collision
can instead result in the wave packet transitioning to a wandering bump itself,
highlighting the dependence of the formation of composite waves on initial
conditions. In this scenario, the bump state does not visit a stable
travelling wave branch and so only transiently adopts a weakly unstable wave
profile before returning to a bump attractor state.

\section{Conclusions} \label{sec:conclusions}
We have provided evidence that the relationship between bump attractors and travelling
waves in a classical network of excitable, leaky integrate-and-fire neurons bears
strong similarities to the one between complex spatiotemporal patterns and waves at
the onset of pipe turbulence. 
We made analytical and numerical progress in the
construction and stability analysis of travelling waves with a large number of
localised spikes, and gained access to their intricate bifurcation structure. This
step was essential, because such waves advect, at low speed, localised patterns that
resemble the bump attractor core. It should be noted that the waves we computed are
only a subset of the ones supported by the model. 

As we completed the present paper, a recent publication~\cite{laing2020moving}
reported the existence of waves with vanishingly small speed, and discontinuous
profiles, in networks of theta neurons, which can be cast as spiking networks with a
polynomial ODE of quadratic type. A natural question
arises as to whether the fluid-dynamical analogy applies in that and other network
models. The level-set approach used in the present paper was particularly effective
because one can
define $m$-spike waves starting from mild solutions to the formal evolution equation
\cref{eq:contMod}, and derive a relatively simple expression for the wave
profile~\cref{eq:nuXi}. While this approach may be harder to carry out in more
detailed spiking models, the general idea of a relationship between localised waves
and bumps in spiking networks could be investigated, by direct simulations, in more
realistic networks (spiking or not). 

An important open question concerns the definition of~\cref{eq:contMod} and, more
generally, of spatially-continuous spiking networks, as dynamical systems posed on
function spaces. This problem has been circumvented here by defining a suitable class
of solutions, introducing the voltage mapping, and then providing proofs of its
relevance to the construction and stability of multiple-spike waves. We believe that
a full dynamical-systems characterisation of similar models will be a key ingredient to
uncover further links between localised waves and bumps in complex,
spatially-extended threshold networks.

\section*{Acknowledgments}
We are grateful to Stephen Coombes, 
Predrag Cvitanovi\'c,
Gregory Faye,
Joel Feinstein,
John Gibson,
Joost Hulshof, 
Rich Kerswell, and
Edgar Knobloch
for insightful discussions.

\appendix

\begin{table}[h!]
  \caption{Parameter descriptions and nominal values for the Discrete
  Integrate-and-Fire Model.}
  \label{tab:parameters}
  \centering
  \begin{tabular}{lcc}
    \toprule
    Parameter                                     & Symbol               & Value(s)  \\ \otoprule
    Number of neurons                             & $n$                  &
    \{80,500,1000,5000\}       \\ 
    Domain half-width                             & $L$                  &  \{1,3,4\} \\
    Synaptic efficacy and time scale              & $\beta$              &  [0,25]    \\
    Synaptic excitation coefficient               & $a_1$                &  11        \\
    Synaptic inhibition coefficient               & $a_2$                &  7         \\
    Synaptic excitation spatial scale             & $b_1$                &  5         \\
    Synaptic inhibition spatial scale             & $b_2$                &  3.5       \\
    Constant external input                       & $I  $                &  0.9       \\
    Time-dependent external input duration        & $\tau_\textrm{ext} $ &  2         \\
    Time-dependent external input strength        & $d_1$                & \{0,2\}    \\
    Time-dependent external input spatial scale   & $d_2$                & \{10,12\}  \\
    \bottomrule
  \end{tabular}
 \end{table}

\section{Synaptic, Reset, and Voltage mappings}\label{supp:SRVMappings}
The following lemma shows that the synaptic contribution is a continuous function on
the plane, hence discontinuities in the voltage come through the reset operator, as
expected. It also provide domains and codomains for the Synaptic, Reset, and Voltage
mappings.
\begin{lemma}\label{supp:prop:SRMappings}
  If \cref{hyp:synapticFunctions} holds, then for the operators $S$, $R$ in \cref{def:SR}
  we have $S \colon C(\RSet) \to BC(\RSet^2)$ and $R \colon C(\RSet) \to B(\RSet^2)$,
  respectively.
\end{lemma}
\begin{proof}
  Fix $u \in C(\RSet)$. The real-valued function $z \mapsto \exp(-z)H(z)$ is
  bounded in $\RSet$, hence $Ru \in B(\RSet^2)$. To prove the result on $S$, we define
  the functions
  \[
    \psi(y,t) = \int_\infty^t e^{z-t} \alpha(z-u(y))\,dz, \qquad
    s(x,t) = (Su)(x,t) = \int_{-\infty}^\infty w(x-y) \psi(y,t)\, dy,
  \]
  and set $K_\alpha = \Vert \alpha \Vert_\infty$, $K_w = \Vert w \Vert_{L^1(\RSet)}$,
  whose existence is guaranteed by \cref{hyp:synapticFunctions}. The function $s$ is
  bounded on $\RSet^2$ because $|\psi| \leq K_\alpha$ on $\RSet^2$, hence
  \[
    |s(x,t)| \leq K_\alpha \int_{-\infty}^\infty
    |w(x-y)|\, dy \leq K_\alpha K_w.
  \]

  In order to prove the continuity of $s$, it is useful to first show that $\psi$ is
  continuous in $t$ on $\RSet$, uniformly in $y$. This claim is proved by noting that, for any $y,t,\tau
  \in \RSet$ with $t < \tau$ we have
  \[
    \begin{aligned}
    | \psi(y,\tau) - \psi(y,t) | 
	&= \bigg| \int_{-\infty}^t [ e^{z-\tau} - e^{z-t}] \alpha(z-u(y)) \, dz + 
		\int_{t}^\tau e^{z-\tau} \alpha(z-u(y)) \, dz  
	  \bigg| \\
	& \leq K_\alpha \int_{-\infty}^t e^{z-t} - e^{z-\tau}\, dz 
	                      + K_\alpha(\tau - t) \\
        & = K_\alpha \big( 1 - e^{-(\tau-t)} + \tau - t \big),
    \end{aligned}
  \]
  which, combined with a similar argument for $\tau < t$, leads to
  \begin{equation}\label{eq:psiBoundApp}
    | \psi(y,\tau) - \psi(y,t) | \leq K_\alpha \big( 1 - e^{-|\tau-t|} + |\tau - t| \big),
    \qquad \text{for all $y,t,\tau \in \RSet$}.
  \end{equation}
  We prove the continuity of $s$ by showing that $|s(\xi,\tau) - s(x,t) | \to 0$ as
  $(\xi,\tau) \to (x,t)$. We consider the following inequality
  \begin{equation}\label{eq:sBoundApp}
    |s(\xi,\tau) - s(x,t) | \leq |s(\xi,\tau) - s(x,\tau) | + |s(x,\tau) - s(x,t) |,
  \end{equation}
  and we note that the second term in the right-hand side of \cref{eq:sBoundApp}
  can be made arbitrarily small as $(\xi,\tau) \to (x,t)$, owing to \cref{eq:psiBoundApp}.
  Therefore it suffices to show that the first term in the right-hand side of
  \cref{eq:sBoundApp} can also be made arbitrarily small as $(\xi,\tau) \to (x,t)$,
  that is, we must show that for any $\eps >0$ there exists $\delta >0$ such that 
  $|s(\xi,\tau) - s(x,\tau) | < \eps$. To prove this statement, we use the
  boundedness of $\psi$ and a change of variables in the integral to obtain the
  estimate
  \[
    \begin{aligned}
    |s(\xi,\tau) - s(x,\tau) | 
    & \leq \int_{-\infty}^\infty | w(y-\xi) - w(y-x)| \, |\psi(y,\tau) | \, dy \\
    & \leq K_\alpha \int_{-\infty}^\infty | w(y+x-\xi) - w(y)| \, dy,  \\
    \end{aligned}
  \]
  therefore, for any $X >0$ we have
  \[
    \begin{aligned}
    |s(\xi,\tau) - s(x,\tau) | 
    & \leq K_\alpha \int_{-\infty}^{-X} | w(y+x-\xi) - w(y)| \, dy  \\
    & + K_\alpha \int_{-X}^{X} | w(y+x-\xi) - w(y)| \, dy  \\
    & + K_\alpha \int_{X}^{\infty} | w(y+x-\xi) - w(y)| \, dy := K_\alpha
    (I_1+I_2+I_3).
    \end{aligned}
  \]
  We bound $I_3$ as follows
  \[
    \begin{aligned}
      I_3 
      & \leq \int_{X}^{\infty} | w(y+x-\xi)| \, dy + \int_{X}^{\infty} | w(y)| \, dy  \\
      & = \int_{X+x-\xi}^{\infty} | w(y)| \, dy + \int_{X}^{\infty} | w(y)| \, dy  
      \leq 2 \int_{X-|x-\xi|}^{\infty} | w(y)| \, dy 
    \end{aligned}
  \]
  and a similar reasoning gives an identical bound for $I_1$,
  \[
      I_1 \leq 2 \int_{X-|x-\xi|}^{\infty} | w(y)| \, dy.
  \]
  We conclude that, for any $X > 0$
  \[
  |s(\xi,\tau) - s(x,\tau) | \leq 4K_\alpha \int_{X-|x-\xi|}^{\infty} | w(y)| \, dy
  + 
    K_\alpha \int_{-X}^{X} | w(y+x-\xi) - w(y)| \, dy. 
  \]
  We now fix $\eps, \delta_1 >0$. Since $w \in L^1(\RSet)$, we can pick $X$ so that
  \[
    |x-\xi| < \delta_1 \qquad \Rightarrow \qquad 4K_\alpha \int_{X-|x-\xi|}^{\infty}
    | w(y)| \, dy < \frac{\eps}{2}.
  \]
  Furthermore, by continuity of $w$ there exists $\delta_2 > 0$ such that
  \[
    |x-\xi| < \delta_2 \qquad \Rightarrow \qquad K_\alpha \int_{-X}^{X} | w(y+x-\xi)
    - w(y)| \, dy \leq \frac{\eps}{2},
  \]
  hence for any $\eps >0$ there exists $\delta = \min(\delta_1,\delta_2) >0$ such
  that $|s(\xi,\tau) - s(x,\tau) | < \eps$, which implies the continuity of $s$. We
  conclude that $S \colon C(\RSet) \to BC(\RSet^2)$.
\end{proof}

\section{Discontinuities of $v_m$}\label{supp:proof:cor:limits}
In some cases it is useful to replace the threshold conditions
\cref{eq:vCrossings,eq:voltageMappingThresholds} by equivalent conditions
involving left limits of the voltage function and mapping, respectively, as specified
by the following result.
\begin{corollary}[Discontinuities of $v_m$]\label{cor:limits}
Under the hypotheses of
\cref{supp:prop:SRMappings}, $v_m = 1$ in $\FSet_\tau$ if, and only if,
  $
    \lim_{\mu \to 0^+} v_m(x,\tau_i(x)-\mu) = v_m(x,\tau_i(x)^-) = 1
  $
  for all $(i,x) \in \NSet_m \times \RSet$.
\end{corollary}
\begin{proof}
  The condition
  $v_m(x,t) = 1$ for $(x,t) \in \FSet_\tau$ is equivalent to
  \[
    v_m(x,\tau_i(x)) = 1 \qquad (i,x) \in \NSet_m \times \RSet.
  \]
  From the defintions of $S$ and $R$, and the continuity $S\tau_i$ on $\RSet^2$ (see
  \cref{supp:prop:SRMappings}) we have, for all $(i,x) \in \NSet_m \times \RSet$
  \[
    \begin{aligned}
    v_m(x,\tau_i(x)) 
    & = I + \sum_{j \in \NSet_m} (S\tau_j)(x,\tau_i(x))  
      - \sum_{j < i} \exp(\tau_j(x) - \tau_i(x)) \\
    & = I + \sum_{j \in \NSet_m} \lim_{\mu \to 0^+} (S\tau_j)(x,\tau_i(x)-\mu) 
      - \sum_{j < i} \exp(\tau_j(x) - \tau_i(x)) \\
    & = I + \lim_{\mu \to 0^+} 
 	  \sum_{j \in \NSet_m} 
          \bigg[ 
	  (S\tau_j)(x,\tau_i(x)-\mu)  
	  -
          (R\tau_j)(x,\tau_i(x)-\mu)  
 	 \bigg] = v_m(x,\tau_i(x)^-),
    \end{aligned}
  \]
  hence $v_m = 1$ in $\FSet_\tau$ if, and only if, 
  $v_m(x,\tau_i(x)^-) = 1$  for all $(i,x) \in \NSet_m \times \RSet$.
\end{proof}

\section{Proof of \cref{prop:nu}}\label{supp:proof:prop:nu}
\begin{proposition}[\tw{m} profile]
  A \tw{m} with speed $c$ satisfies $(V_m\tau)(x,t)= \nu_m(ct-x;c,T)$,
  and its $(c,T)$-dependent travelling wave profile $\nu_m$ is given by
  \begin{equation} \label{eq:nuXiSupp}
  \begin{aligned}
    \nu_m(\xi;c,T) 
    = I  & - \sum_{j \in \NSet_m}
    \exp\bigg( -\frac{\xi -cT_j}{c} \bigg) H\bigg(\frac{\xi- c T_j}{c}\bigg) \\ 
    & + \frac{1}{c} \sum_{j \in \NSet_m}
    \int_{-\infty}^\xi \exp\bigg( \frac{z-\xi}{c} \bigg) \int_0^\infty
    w(y-z+cT_j) p(y/c) \, dy \,dz.
   \end{aligned}
   \end{equation}
\end{proposition}
\begin{proof} 
  We set $\tau_j(x) = x/c + T_j$ for $j \in \NSet_n$. The first sum
  in~\cref{eq:nuXiSupp}
  is immediate, as
  \[
   (R\tau_j)(x,t) = \exp\bigg( \frac{cT_j - ct + x}{c} \bigg) H\bigg(\frac{ct -x -
   c T_j}{c}\bigg)
  \]
  The second sum is obtained as follows:
\[
 \begin{aligned}
   (S\tau_j)(x,t) 
    & = \int_{-\infty}^t \int_{-\infty}^\infty \exp(s-t)
    w(x-y) \alpha(s-\tau_j(y)) \, dy \, ds \\
    & = \int_{-\infty}^t \int_{0}^{\infty} \exp(s-t)
    w(x+y-cs+cT_{j}) p(y/c) \, dy \, ds \\
    & =\frac{1}{c} 
    \int_{-\infty}^{ct-x} \int_{0}^{\infty}\exp\bigg(\frac{s+x-ct}{c}\bigg) w(y-s+cT_{j})
    p(y/c) \, dy \, ds .
  \end{aligned}
 \]
\end{proof}

\section{Travelling wave stability} \label{supp:sec:linearVM}

We begin by showing that if two distinct $m$-spike solutions have firing functions
$\tau$ and $\tau+\phi$, respectively, then the perturbations $\phi$ satisfy a linear
equation to leading order. The following lemma also specifies admissible
perturbations, namely $\phi$ are in the Banach space $C_\eta(\RSet,\RSet^m)$:
perturbations are allowed to grow exponentially as $|x| \to
\infty$, at a rate at most equal to $\eta$, which bounds the decay rate of the
connectivity kernel function $w \in L^1_\eta(\RSet)$.

\begin{lemma}[Linearisation of the voltage mapping operator]\label{lemma:linearVM}
  Assume \cref{hyp:synapticFunctions}, and let $(c,T)$ be the coarse variables of a
  \tw{m} with firing functions $\tau$. Further, let $L$ be the linear operator defined
  by $L\phi = \big( (L\phi)_1,\ldots,(L\phi)_m \big)$, where
  \[
    (L\phi)_i =\sum_{j \in \NSet_m} (\phi_i - \phi_j) 1_{j<i} 
    + \int_{cT_{ji}}^\infty
    e^{-y/c} w(y)  \psi_{ij}(y) \big[ \phi_i - \phi_j(\blank - y) \big] \, dy,
    \quad i \in \NSet_m,
  \]
  with coefficients $T_{ij}$ and functions $\psi_{ij}$ given by
  \[
    T_{ij} = T_i - T_j, \qquad \psi_{ij} \colon 
    [cT_{ij},\infty) \to \RSet, \quad y \mapsto p(0) + \int_0^{y/c-T_{ji}}
    e^s p'(s)\, ds, \qquad i,j \in \NSet_n,
  \]
  respectively. The following statements hold:
  \begin{enumerate}
    \item $L$ is a bounded operator from $C_\eta(\RSet,\CSet^m)$ to itself.
    \item Let $0 < \eps \ll 1$ and $\phi \in C_\eta(\RSet,\RSet^m)$. If $\tau + \eps
      \phi$ are firing functions of an $m$-spike CIFM solution (a perturbation of the \tw{m}), then
  \begin{equation}\label{supp:eq:LPhi}
    0 = L \phi + O(\eps) \qquad \text{in $\RSet$}
  \end{equation}
  \end{enumerate}
\end{lemma}
\begin{proof}
  \emph{Part 1}. If $\phi \in C_\eta(\RSet,\CSet^m)$, then $\phi \in
  C(\RSet,\CSet^m)$ and $L\phi \in C(\RSet,\CSet^m)$. We
  show that for any $\phi \in C_\eta(\RSet,\CSet^m)$ there exists a positive constant
  $\kappa_{m,\eta}$ such that $ \Vert L\phi \Vert_{C_{m,\eta}}\leq \kappa_{m,\eta}
  \Vert \phi \Vert_{C_{m,\eta}}$, which implies that $L$ is
  a bounded operator from $C_\eta(\RSet,\CSet^m)$ to itself. We begin by estimating
  $\psi_{ij}$: by \cref{hyp:synapticFunctions} there exist constants $K_p$, $K_{p'}$
  such that
  \[
    \begin{aligned}
    | \psi_{ij}(y) | & \leq |p(0)| + \int_0^{y/c-T_{ji}} e^z |p'(z)|\, dz  \\
		     & \leq K_p + K_{p'}\int_0^{y/c-T_{ji}} e^z \, dz  
		     = K_p + K_{p'}(e^{y/c-T_{ji}}-1),
    \end{aligned}
  \]
  therefore, introducing the constant $K = 2\max(K_p, K_{p'}) \max_{i,j}
  e^{-T_{ji}}$,
  \begin{equation}\label{supp:eq:PsiEst}
      e^{-y/c} | \psi_{ij}(y) | \leq e^{-y/c} ( K_p + K_{p'}e^{y/c-T_{ji}}) \leq K
  \end{equation}
  uniformly in $(i,j,y) \in \NSet_m \times \NSet_m \times [cT_{ji},\infty)$.

  We now fix $\phi \in C_\eta(\RSet,\CSet^m)$, $x \in \RSet$, and estimate
  \begin{equation}\label{supp:eq:LPhiEst}
    \begin{aligned}
    |(L\phi)(x)|_{\infty} 
    & \leq 
        \max_{i \in \NSet_m} \sum_{j \in \NSet_m} \Big( | \phi_i(x) | + | \phi_j(x)
	| \Big)  \\
    & +
        \max_{i \in \NSet_m} \sum_{j \in \NSet_m} 
    \int_{cT_{ji}}^\infty
  e^{-y/c} | w(y)  \psi_{ij}(y)|\; | \phi_i(x) - \phi_j(x - y)| \, dy.
    \end{aligned}
  \end{equation}
  For the first summands in \cref{supp:eq:LPhiEst} we find 
  \begin{equation}\label{supp:eq:firstSum}
    | \phi_i(x) | + | \phi_j(x) | \leq 2 |\phi(x)|_{\infty} \leq 2 e^{\eta |x|}
    \Vert \phi \Vert_{C_{m,\eta}} 
  \end{equation}
  For the second summands in \cref{supp:eq:LPhiEst}, we estimate
  \begin{equation}\label{supp:eq:secondSum}
    \begin{aligned}
    \int_{cT_{ji}}^\infty e^{-y/c} 
    &
    | w(y)  \psi_{ij}(y)|\; | \phi_i(x) - \phi_j(x - y)| \, dy 
		     && \qquad \text{(by \cref{supp:eq:PsiEst})} \\
    & \leq K \int_{cT_{ji}}^\infty 
                     | w(y) |\, |\phi_i(x) - \phi_j(x - y)| \, dy  \\
    & \leq K \Vert \phi \Vert_{C_{m,\eta}} \int_{cT_{ji}}^\infty
		|w(y)| \big( e^{\eta |x|} + e^{\eta|x-y|} \big)\, dy \\
    & \leq K e^{\eta |x|} \Vert \phi \Vert_{C_{m,\eta}} 
       \int_{cT_{ji}}^\infty |w(y)| \big( 1 + e^{\eta |y|} \big)\, dy \\
    & \leq K e^{\eta |x|} \Vert \phi \Vert_{C_{m,\eta}}
      \bigg( \Vert w \Vert_{L^1_\eta} +
	      \int_{cT_{ji}}^\infty
	      |w(y)| e^{\eta |y|}\, dy 
      \bigg)
                             && \qquad \text{($w$ even)} \\
    & \leq K e^{\eta |x|} \Vert \phi \Vert_{C_{m,\eta}}
      \bigg( \Vert w \Vert_{L^1_\eta} +
	2 \int_{-\infty}^\infty |w(y)| e^{\eta y}\, dy 
      \bigg)
      \\
    & \leq 3 K \Vert w \Vert_{L^1_\eta} e^{\eta |x|} \Vert \phi \Vert_{C_{m,\eta}}
    \end{aligned}
  \end{equation}
  Combining \crefrange{supp:eq:LPhiEst}{supp:eq:secondSum} we obtain
  \[
    \begin{aligned}
      \Vert L \phi \Vert_{C_{m,\eta}} 
      & = \sup_{x \in \RSet} e^{-\eta |x|} |L\phi(x)|_{\infty} \\
      & \leq m(2 + 3K \Vert w \Vert_{L^1_\eta}) \Vert \phi \Vert_{C_{m,\eta}}
      & := \kappa_{m,\eta} \Vert \phi \Vert_{C_{m,\eta}},
    \end{aligned}
  \]
  which concludes the proof of part 1.

  \emph{Part 2}. We set $u = \tau + \eps \phi \in C_\eta(\RSet,\RSet^m)$, for $0 <
  \eps \ll 1$. By
  main hypothesis $u$ and $\tau$ are firing functions of two distinct $m$-spike
  CIFM solutions. We claim that this implies \Cref{supp:eq:LPhi}. Indeed, since $u$ is a
  firing function, then $V_m u = 1$ on $\FSet_u$, that is
  \begin{equation}\label{supp:eq:step1}
    1 = I + \sum_{j \in \NSet_m} 
    \Big(
    (Su_i)(x,u_j(x)) + (Ru_j)(x,u_i(x)) 
    \Big),
    \qquad (i,x) \in \NSet_m \times \RSet.
  \end{equation}
  We obtain
  \[
    \begin{aligned}
    (Su_j)(\blank,u_i) 
    & 
    = \int_{-\infty}^\infty w(\blank - y) \int_{-\infty}^0 e^z
				    \alpha(z+u_i - u_j(y)) \, dz \, dy \\
    &
    = (S\tau_j)(\blank,\tau_i) + \eps \int_{-\infty}^\infty w(\blank -y) 
	  \big( \phi_i - \phi_j(y) \big) 
	  \int_{-\infty}^0 e^z \alpha'(z+\tau_i - \tau_j(y)) \, dz \, dy \\
	  & + O(\eps^2),
    \end{aligned}
  \]
  where we have denoted by $\alpha' = p' H + p \delta$ the distributional derivative
  of $\alpha$. We now manipulate the integral in the previous equation as follows
  \[
    \begin{aligned}
    \int_{-\infty}^\infty 
         &
	  w(\blank -y) 
	  \big( \phi_i - \phi_j(y) \big) 
	  \int_{-\infty}^0 e^z \alpha'(z+\tau_i - \tau_j(y)) \, dz \, dy  \\
    &
    =
    \int_{-\infty}^\infty w(\blank -y) 
	  \big( \phi_i - \phi_j(y) \big) 
	  \int_{-\infty}^{\tau_i -\tau_j(y)} e^{z+ \tau_i - \tau_j(y)} \alpha'(z) 
								\, dz \, dy  \\
    &
    =
    \int_{-\infty}^\infty w(\blank -y) 
	  \big( \phi_i - \phi_j(y) \big) 
	  \int_{-\infty}^{\tau_i -\tau_j(y)} e^{z+ \tau_i - \tau_j(y)} 
	  \big( p'(z)H(z) + p(z) \delta(z) \big)
								\, dz \, dy  \\
    &
    =
    \int_{-\infty}^{\tau^{-1}_j(\tau_i)}
    w(\blank -y) 
	  \big( \phi_i - \phi_j(y) \big) 
	  e^{\tau_j(y) - \tau_i} 
	  \Big(
	    p(0) + \int_0^{\tau_i-\tau_j(y)} 
	    e^z p'(z) \, dz
	  \Big) \,dy \\
  \end{aligned},
  \]
  hence for all $i,j \in \NSet_m$ we obtain
  \[
    (Su_j)(\blank,u_i)  = (S\tau_j)(\blank,\tau_i)  
    + \eps e^{T_{ji}} 
    \int_{-\infty}^\infty e^{-y/c} w(y)  \psi_{ij}(y) 
    \big[ \phi_i - \phi_j(\blank- y) \big] \,dy + O(\eps^2)
  \]
  For the reset operator, we obtain, for all $i,j \in \NSet_m$
  \[
  \begin{aligned}
    Ru_j(\blank,u_i^-) & = -\lim_{\kappa \to 0^+} \exp(-u_i + \kappa + u_j) H(u_i -\kappa - u_j) \\
    & = R \tau_j(\blank,\tau_i^-) + \eps \lim_{\kappa \to 0^+} \exp(-T_{ij} + \kappa) (\phi_i - \phi_j) H(T_{ij}) + O(\eps^2) \\
    & = R \tau_j(\blank,\tau_i^-) + \eps \exp(T_{ji}) (\phi_i - \phi_j) 1_{j<i} +
    O(\eps^2).
  \end{aligned}
  \]
  Combining \cref{supp:eq:step1} with the expansions obtained for $S$ and $R$, exploiting
  the condition $V_m\tau = 1$ on $\FSet_\tau$, and dividing by $\eps e^{T_{ji}}$ we
  obtain
  \[
    0 = (L\phi)_i + O(\eps) \qquad \text{on $\RSet$, for all $i \in \NSet_m$}.
  \]
  which implies \Cref{supp:eq:LPhi}. 
\end{proof}

\begin{remark}
  Note that the operator $L$ depends on the coarse variables $(c,T)$, albeit we omit
  this dependence for notational simplicity.
\end{remark}

We are now ready to define linear stability for a \tw{m}, which we
adapt from \cite{Bressloff:2000dqrep}. Intuitively, we compare the firing set
$\FSet_\tau$ of a \tw{m} with the firing set $\FSet_{\tau+\phi}$ of a perturbed
$m$-spike solution with $\Vert \phi \Vert_{C_{m,\eta}} \ll 1$, for which $\phi$
satisfy \cref{supp:eq:LPhi} to leading order. If the sets $\FSet_\tau$ and $\FSet_{\tau+\phi}$
are close around $t=0$ and remain close for all positive times, we deem the wave
linearly stable. With reference to 
\cref{fig:TildeTau}(c), we observe that,
when \tw{m} crosses the axis $t=0$, each one of its firing functions $\tau_i$ is
perturbed by an amount $\phi_i(-cT_i)$. Roughly speaking, a \tw{m} is linearly stable
if $\phi_i(-cT_i)$ being small implies that $\phi_i(x)$ stays small for all $x \in
(-cT_i,\infty)$ and $i \in \NSet_m$. If a wave is linearly stable and all $\phi_i$
decay to $0$ as $x \to \infty$ we say that the wave is asymptotically linearly
stable. More precisely:

\begin{definition}[Linear stability of \tw{m}]
  \label{def:linearStability}
  A \tw{m} with coarse varaibles $(c,T)$ is linearly stable to perturbations $\phi$
  if $\phi \in \ker L$, and for each $\eps >0$ there exists $\delta = \delta(\eps)
  >0$, such that if $|\phi_i(-cT_i)| < \delta$, then $|\phi_i(x)| < \eps$
  for all $(i,x) \in \NSet_m \times (-cT_i,\infty)$.

  A \tw{m} is asymptotically linearly stable to perturbations $\phi$ if is linearly
  stable to perturbations $\phi$ and $| \phi (x) |_{\infty} \to 0$ as $x \to
  \infty$.
\end{definition}

We have seen that a \tw{m} can be constructed by solving a nonlinear problem in the
unknowns $(c,T)$. The following lemma, which is the central result of this section,
establishes that linear stability of a \tw{m} with respect to exponential
perturbations of the firing functions can also be determined by finding roots
of a $(c,T)$-dependent, complex-valued function.

%

\begin{lemma}[\tw{m} stability]\label{lem:twStability}
  Assume \cref{hyp:synapticFunctions}, let $(c,T)$ be coarse variables of a
  \tw{m}, and let $\DSet_{a,b} = \{ z \in \CSet \colon a \leq \real z \leq b
  \}$. Further, let $E$ be the complex-valued function
  \begin{equation}\label{eq:EDefinition}
    E \colon \DSet_{-\eta,\eta} \to \CSet, \quad z \mapsto \det[ D - M(z) ],
  \end{equation}
  where 
  $M \in \CSet^{m \times m}$, $D = \diag(D_1, \ldots,D_m) \in \RSet^{m \times m}$, 
  are the matrices with elements
  \[
    M_{ij}(z)  = e^{T_{ji}} 
    \bigg[
      1_{j<i} + \int_{cT_{ji}}^\infty e^{-(z+1/c)y}  w(y) \psi_{ij}(y)\, dy 
    \bigg],
    \qquad
    D_i = \sum_{k \in \NSet_m} M_{ik}(0),
  \]
  respectively, then:
  \begin{enumerate}
    \item If $\lambda$ is a root of $E$, then its complex conjugate $\lambda^*$ is
      also a root of $E$, and there exists a nonzero $\Phi \in \ker[
      D-M(\lambda)]$ such that $\Phi e^{\lambda x}, \Phi^* e^{\lambda^* x} \in
      \ker{L}$, where $L$ is defined as in \cref{lemma:linearVM}.
    \item $E$ has a root at $0$. \tw{m} is linearly stable (but not
      asymptotically linearly stable) to perturbations 
      $\phi \colon x \mapsto \kappa v$, where $\kappa \in \RSet \setminus \{0\}$ and $v=(1,\ldots,1) \in
      \RSet^m$.
     
    \item If $\lambda$ is a root of $E$ in $\DSet_{-\eta,0} \setminus \iunit \RSet$,
      then $\tw{m}$ is linearly asymptotically stable to perturbations $\Phi
      e^{\lambda x} + \Phi^* e^{\lambda^* x}$.
  \end{enumerate}
  \end{lemma}
%
  \begin{proof}
    \emph{Part 1}. We observe that $D$ has purely real entries, and a
    direct calculation shows $M(z^*) = M^*(z)$. If $E(\lambda)=0$, then there
    exists $\Phi \in \CSet^m \setminus \{0\}$ such that $D \Phi = M(\lambda) \Phi$, that
    is, $\Phi \in \ker[ D - M(\lambda)]$. Taking the complex conjugate we obtain
    $D \Phi^* = M^*(\lambda) \Phi^* = M(\lambda^*) \Phi^*$, hence $E(\lambda^*) = 0$,
    therefore $\lambda^*$ is also a root. 
    
    We now set $\phi = \Phi e^{\lambda x}$ which is in $C_\eta(\RSet,\CSet^m)$
    because $\lambda \in \DSet_{-\eta,\eta}$, and we obtain 
    \[
      (L\phi)(x) = e^{\lambda x} [D - M(\lambda)]\Phi =0
      \qquad x \in \RSet
    \]
    because $\Phi \in \ker[ D - M(\lambda) ]$. The previous identity implies $\Phi
    e^{\lambda x}, \Phi^* e^{\lambda^* x} \in \ker{L}$.

    \emph{Part 2}. 
    By definition of $D$ and $M$ we have $ [ D - M(0)]v=0 $, hence $v$ is in
    the kernel of $D-M(0)$ and
    $E(0) = 0$. We fix $\kappa \in \RSet \setminus \{0\}$, use part 1 with $\lambda = 0$,
    $\Phi = \kappa v$, and deduce that the mapping $\phi \colon x \mapsto
    \kappa v$, which is an element of $C_\eta(\RSet,\RSet^m)$, is in $\ker L$. Since
    $\phi_i(x) \equiv \kappa$, for all $i \in \NSet_m$, $\tw{m}$ is linearly stable
    according to \cref{def:linearStability}. However, $|\phi(x)|_\infty \to \kappa
    \neq 0$ as $x \to \infty$, so \tw{m} is not asymptotically linearly stable.

    \emph{Part 3}. Let $\lambda= \mu + \iunit \omega$. By main hypothesis $\mu <0$.
    From part 1 we deduce that there exists $\Phi \in
    \ker[D-M(\lambda)]$, such that $\phi(x) = \Phi e^{\lambda x} + \Phi^*
    e^{\lambda^* x} \in \ker L$. We note that $|\Phi|_\infty$ can be fixed to an
    arbitrary nonzero constant, and we bound $\phi_i$ as follows
    \begin{equation}\label{supp:eq:phiBound}
      |\phi_i(x)| \leq 2 | \Phi |_\infty e^{\mu x} \leq K | \Phi |_\infty, 
      \quad K = 2 \max_{j \in \NSet_m} e^{-cT_j \mu},
      \quad (i,x) \in \NSet_m \times [-cT_i,\infty). 
    \end{equation}
    We now fix $\eps>0$. The bound \cref{supp:eq:phiBound} and the choices $\delta = \eps$
    and $| \Phi |_\infty < \eps/K$ imply that \tw{m} is linearly stable, according to
    \cref{def:linearStability}. Using again \cref{supp:eq:phiBound} we obtain
    \[
      \lim_{x \to \infty} | \phi(x) |_{\infty}  
      \leq 
      2 |\Phi|_{\infty} \lim_{x \to \infty} e^{\mu x} = 0
    \]
    therefore \tw{m} is linearly asymptotically stable.
\end{proof}

\Cref{lem:twStability} provides a link between exponential perturbations to the
firing times of a \tw{m} and zeroes of the function $E$ in the strip
$\DSet_{-\eta,\eta} \subset \CSet$. The function $E$ depends on $(c,T)$ via the
entries of the matrices $D, M$, and can be evaluated numerically at each point
$z \in \DSet_{-\eta,\eta}$. 

In PDEs, linear stability of a travelling wave is determined by the spectrum of a
linear operator, which contains a $0$ eigenvalue corresponding to a translational
perturbation mode. 
Part 2 of \Cref{lem:twStability} provides an analogous result
for a \tw{m}, which is linearly stable, but not asymptotically linearly stable
(therefore neutrally stable), to perturbations that shift the firing functions
homogeneously. Part 3 of \Cref{lem:twStability} suggests
that a \tw{m} is stable if all nonzero roots of $E$ have strictly negative real
parts. Initial guesses for the roots can be obtained by plotting $0$-level sets of
the function $E$, for fixed $(c,T)$.


\section{Two-parameter continuation of bifurcations}
\label{sec:supp:TwoParameterContinuation}
Grazing points are found generically as a secondary control parameter, say
$\gamma$, is varied. It is possible to perform a $2$-parameter continuation
of the grazing point in the $(\beta,\gamma)$-plane by continuing in
$\gamma$ solutions the following problem:
\begin{problem}[Grazing point computation]\label{prob:G}
  Find $(c,T_1,\ldots,T_m,T_{G},\beta_{G}) \in \RSet_{>0} \times \RSet^{m+2}$
  such that  $T_1< \cdots < T_m < T_{G}$ and
  \begin{align}
    & T_1 = 0, \\
    & \nu([ cT_i ]_-; c, T_1, \ldots, T_m,\beta_G) = 1, \quad i \in \NSet_m,\\
    & \nu(cT_G;c, T_1, \ldots, T_m,\beta_G) = 1,\\
    & \nu'(cT_G; c, T_1, \ldots, T_m,\beta_G) = 0.
  \end{align}
\end{problem}
We note that we have exposed the dependence of $\nu$ and $\nu'$ on $\beta$ in the
previous problem. Unlike the numerical continuation presented in the main text, here
$\beta$ is a free parameter, which is determined also by Newton's method. 

Similarly, we can trace loci of Hopf bifurcations in the $(\gamma, \beta)$-plane, by
solving
\begin{problem}[Computation of Hopf bifurcations]\label{prob:Hopf}
  Find $(c,T_1,\ldots,T_m,\beta_\textrm{HB},\omega_\textrm{HB})  \in
  \RSet_{> 0} \times \RSet^{m+2}$ such that $T_1< \cdots < T_m$ and
  \begin{align}
    & T_1 = 0, \\
    & \nu([ cT_i ]_-; c, T_1, \ldots, T_m,\beta_\textrm{HB}) = 1, \quad i \in \NSet_m,\\
    & E(i\omega_\textrm{HB};c,T_1,\ldots,T_m,\beta_\textrm{HB})= 0.
  \end{align}
\end{problem}

\begin{figure} \label{fig:sup:purelyexcitatory_bif_profiles}
  \centering
  \includegraphics{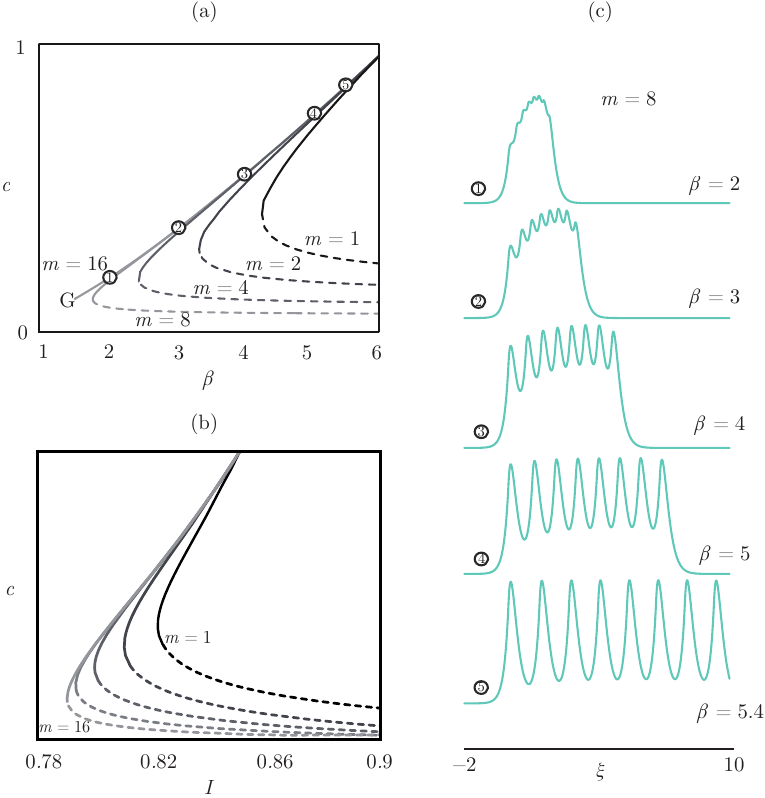}
 \caption{Bifurcation diagram showing wave speed branches $c$ for \tw{m} states
 with $m = 1,2,4,8,16$, with purely excitatory kernel, $a_1=2$, $b_1=5$, $a_2=0$.
 (a) External current is fixed at $I=0.82$ and solutions
 continued in $\beta$, here $m = 1, 2, 4, 8$ branches switch stability at a
 fold, but for $m = 16$ the branch terminates at a grazing point as seen in the
 lateral-inhibition case previously. (b) The synaptic processing rate is fixed at
 $\beta = 4$ and the waves continued in $I$. Note that in both cases a greater $m$ leads to
 larger speeds. (c) Synaptic profiles for \tw{8} solutions labelled (1--5) in (a).
 Smaller $\beta$ values result in more localised profiles.}
\end{figure}

\bibliography{references}
\bibliographystyle{siamplain}

\end{document}